\documentclass[reqno,12pt]{amsart}
\usepackage{amsfonts}
\usepackage{bbm}
\usepackage{}
\numberwithin{equation}{section}
\usepackage{color}
\usepackage{graphicx}
\usepackage{indentfirst}
\usepackage{color}
\usepackage{amssymb}
\usepackage{mathrsfs}
\usepackage{xy}
\usepackage{tikz}

\xyoption{all}
\def\Ext{\mbox{\rm Ext}\,} \def\Hom{\mbox{\rm Hom}} \def\dim{\mbox{\rm dim}\,} 
    \def\mod{\mbox{\rm \textbf{mod}}\,}
    
\def\cone{\mbox{\rm cone}}\def\cocone{\mbox{\rm cocone}}
\def\ind{\mbox{\rm ind}\,}\def\cim{\mbox{\rm sim}}\def\add{\mbox{\rm add}}
 
\def\Filt{\mbox{\rm \textbf{Filt}}}\def\supp{\mbox{\rm supp}}
\def\f-tors{\mbox{\rm f-tors}} \def\s\tau-tilt{\mbox{\rm s\tau-tilt}}
 \def\Binv{\mbox{\rm Binv}}\def\inv{\mbox{\rm inv}}\def\rep{\mbox{\rm rep}}\def\coker{\mbox{\rm coker}}

\theoremstyle{plain}
\newtheorem{theorem}{\bf Theorem}[section]
\newtheorem{lemma}[theorem]{\bf Lemma}
\newtheorem{corollary}[theorem]{\bf Corollary}
\newtheorem{proposition}[theorem]{\bf Proposition}

\theoremstyle{definition}
\newtheorem{definition}[theorem]{\bf Definition}
\newtheorem{remark}[theorem]{\bf Remark}
\newtheorem{example}[theorem]{\bf Example}

\newcommand{\bt}{\begin{theorem}}
\newcommand{\et}{\end{theorem}}
\newcommand{\bl}{\begin{lemma}}
\newcommand{\el}{\end{lemma}}
\newcommand{\bd}{\begin{definition}}
\newcommand{\ed}{\end{definition}}
\newcommand{\bc}{\begin{corollary}}
\newcommand{\ec}{\end{corollary}}
\newcommand{\bp}{\begin{proof}}
\newcommand{\ep}{\end{proof}}
\newcommand{\bx}{\begin{example}}
\newcommand{\ex}{\end{example}}
\newcommand{\br}{\begin{remark}}
\newcommand{\er}{\end{remark}}
\newcommand{\be}{\begin{equation}}
\newcommand{\ee}{\end{equation}}
\newcommand{\ba}{\begin{align}}
\newcommand{\ea}{\end{align}}
\newcommand{\bn}{\begin{enumerate}}
\newcommand{\en}{\end{enumerate}}
\newcommand{\bcs}{\begin{cases}}
\newcommand{\ecs}{\end{cases}}

\newcommand{\RNum}[1]{\uppercase\expandafter{\romannumeral #1\relax}}

\makeatletter
\renewcommand{\section}{\@startsection{section}{1}{0mm}
  {-\baselineskip}{0.5\baselineskip}{\bf\leftline}}
\makeatother

\begin{document}
\title[semibricks, torsion-free classes and the Jordan-H\"{o}lder property]{semibricks, torsion-free classes and\\ the Jordan-H\"{o}lder property}
\author[L. Wang, J. Wei, H. Zhang, P. Zhang]{Li Wang, Jiaqun Wei, Haicheng Zhang, Peiyu Zhang}
\address{School of Mathematical and Physics, Anhui Polytechnic University, Wuhu, 241000, P. R. China. \endgraf}
\email{wl04221995@163.com (L. Wang); zhangpy@ahpu.edu.cn.(P. Zhang).}
\address{Institute of Mathematics, School of Mathematical Sciences, Nanjing Normal University,
 Nanjing 210023, P. R. China.\endgraf}
\email{weijiaqun@njnu.edu.cn (J. Wei); zhanghc@njnu.edu.cn (H. Zhang).}

%\thanks{$\ast$: Corresponding author. The authors are supported by the Natural Science Foundation of China (Grant No.11471269; No.211070B31704).}

\begin{abstract}
Let $\mathscr{C}$ be an extriangulated category and
$\mathcal{X}$ be a  semibrick in $\mathscr{C}$. Let $\mathcal{T}$ be the filtration subcategory generated by $\mathcal{X}$. We introduce the weak Jordan-H\"{o}lder property {(WJHP)} and Jordan-H\"{o}lder property {(JHP)} in $\mathscr{C}$ and show that $\mathcal{T}$ satisfies {(WJHP)}. Furthermore, $\mathcal{T}$ satisfies {(JHP)} if and only if $\mathcal{X}$ is proper. Using reflection functors and $c$-sortable elements, we give a combinatorial criterion for the torsion-free class satisfying  {(JHP)} in the representation category of a quiver of type $A$.\end{abstract}

\maketitle

%%%%%%%%%%%%%%%%%%%%%%%%%%%%%%%%%%%%%%%%%%%%%%%%%%%%%%%%%%%%%%%%%%%%%%%%%%%%%%%%%%%%%%%%%%%%%%%%%%%%%%%%%%%%%%%%%%%%%%%%%%%%%%%%%%%%%%%%%%%%%%%%%%%%%%%%
\section{Introduction}
It is well-known that the category of modules with finite length over a ring satisfies the {Jordan-H\"{o}lder property}, abbreviated by (JHP), i.e., for any module of finite length, all its composition series are equivalent (cf. \cite[Theorem I.1.2]{Aus}). It is also known that an abelian category satisfies (JHP) if every object has finite length (cf. \cite{St}). A natural question is to investigate to what extent this property is valid in various categories.

Clearly, the notions of {simple objects}, {composition series} and {lengths} of objects are fundamental in studying (JHP). These notions for exact categories were introduced in \cite{BHL}. In \cite{Br}, the notions of {diamond exact categories} and {Artin-Wedderburn exact categories} were introduced and it was proved that these categories satisfy (JHP). However, there exist many exact categories which do not satisfy (JHP), this is mainly because the (length of) composition series may not be unique (cf. \cite{En2}).
Enomoto \cite{En2} studied (JHP) in exact categories, which serve as a useful categorical framework for studying various subcategories of module categories. The notion of {Grothendieck monoids} of exact categories was introduced by Berenstein and Greenstein in \cite{Be} with aim to study Hall algebras of exact categories. Enomoto \cite{En2} used the Grothendieck monoids to give some simple characterizations of (JHP) in exact categories. In particular, he showed that a skeletally small exact category $\mathcal {E}$ with a finitely generated Grothendieck group $K_0(\mathcal {E})$ satisfies (JHP) if and only if $K_0(\mathcal {E})$ is a free abelian group with rank equal to the number of non-isomorphic simple objects in $\mathcal {E}$.

According to Schur's lemma, for any morphism $f$ between two simple objects in an abelian category, $f$ is either zero or an isomorphism. In other words, the set of simple objects in an abelian category forms a semibrick. We recall that an object in an additive category $\mathcal{C}$ is called a {brick}, if its endomorphism ring is a division ring; A set $\mathcal{X}$ of isoclasses (isomorphism classes) of bricks in $\mathcal{C}$ is called a {semibrick} if $\Hom_{\mathcal{C}}(X_1,X_2)=0$ for any two non-isomorphic objects $X_1,X_2$ in $\mathcal{X}$. Schur's lemma has been generalised to exact categories in \cite{En1}. Given an exact category $\mathcal {E}$, a {wide subcategory} $\mathcal {W}$ is an extension-closed exact abelian subcategory of $\mathcal {E}$. In addition, if $\mathcal {W}$ is a length abelian category, then it is called a {length wide subcategory}
of $\mathcal {E}$. By considering the {filtration subcategories}, Enomoto \cite{En1} showed that there exists a bijection between semibricks and length wide subcategories in $\mathcal {E}$, which is a generalization of the classical result of Ringel \cite{Ringel}.
Recently, Nakaoka and Palu \cite{Na} introduced an extriangulated category by extracting properties on triangulated categories and exact categories. The bijection between semibricks and length wide subcategories was generalised to extriangulated categories in \cite{Wa}, where some properties of the filtration subcategory $\Filt_{\mathscr{C}}(\mathcal{X})$ generated by a semibrick $\mathcal{X}$ in an extriangulated category $\mathscr{C}$ have been given.

It is interesting to investigate (JHP) in extriangulated categories.
Note that for an abelian category $\mathcal{A}$ with the set $\cim(\mathcal{A})$ of isoclasses of simple objects, $\mathcal{A}$ satisfies (JHP) if and only if $\mathcal{A}=\Filt_{\mathcal{A}}(\cim(\mathcal{A})$. The proof is by means of {intersection} or {lattice-theoretic} (cf. \cite{Br}, \cite[p.92]{St}). Clearly, $\cim(\mathcal{A})$ is a semibrick. Therefore, we intend to investigate (JHP) in the filtration subcategories generated by semibricks in extriangulated categories.
In fact, we introduce the weak Jordan-H\"{o}lder property for extriangulated categories, abbreviated by (WJHP). We firstly prove that the filtration subcategory $\mathcal{T}=\Filt_{\mathscr{C}}(\mathcal{X})$ generated by any semibrick $\mathcal{X}$ satisfies (WJHP), furthermore, $\mathcal{T}$ satisfies (JHP) if and only if $\mathcal{X}$ is a proper semibrick.

The torsion-free classes in the representation theory of algebras are central and have relations with the tilting theory. Enomoto \cite {En3} classified all torsion-free classes in any length abelian categories by using bricks, and in \cite{En2} he also gave a criterion whether a functorially finite torsion-free class satisfies (JHP) by using the language of $\tau$-tilting theory. Moreover, by using the combinatorics of the symmetric group, Enomoto \cite{En2} investigated the simple objects and (JHP) in a torsion-free class of the category $\rep(Q)$ of representations of a quiver $Q$ of type $A$.
He proved that if $Q$ is linearly oriented, then every torsion-free class in $\rep(Q)$ satisfies (JHP). An interesting question is whether there is a formula to count the number of torsion-free classes satisfying  (JHP) in the representation category of any $A$-type quiver $Q$ (cf. \cite[Problem9.7]{En2}). To do this, it is essential to give a simple criterion for the torsion-free classes satisfying (JHP) in $\rep(Q)$. Using reflection functors and $c$-sortable elements, we give a combinatorial criterion.

The paper is organized as follows: We recall some definitions and properties of extriangulated categories and filtration subcategories in Section 2. In Section 3,  we
introduce (WJHP), (JHP) and proper semibricks in  extriangulated categories, and show that the filtration subcategory $\mathcal{F}(\mathcal{X})$ generated by a semibrick $\mathcal{X}$ satisfies (WJHP). Furthermore, $\mathcal{F}(\mathcal{X})$ satisfies (JHP) if and only if $\mathcal{X}$ is proper. Finally, we give a combinatorial criterion for the torsion-free classes satisfying (JHP) in the representation category of a quiver of type $A$ in Section 4.

Throughout this paper, we assume, unless otherwise stated, that all considered categories are skeletally small, Krull--Schmidt and subcategories are {full} and closed under isomorphisms. For an additive category $\mathscr{C}$, we denote by ${\rm ind}(\mathscr{C})$ a complete set of indecomposable objects in $\mathscr{C}$; for a set $\mathcal {X}$ of objects in $\mathscr{C}$, we denote by ${\rm add}\mathcal {X}$ the subcategory of $\mathscr{C}$ consisting of direct summands of finite direct sums of objects in $\mathcal {X}$.  Let $k$ be a field and $Q$ be a finite acyclic quiver, for a $k$-representation $M$ of $Q$, we denote by $M_i$ the vector space corresponding to the vertex $i$; we denote by $\rep(Q)$ the category of finite dimensional representations of $Q$, which is identified with the category $\mod kQ$ of finite dimensional left $kQ$-modules; we denote by $S_i$ the one-dimensional simple representation associated to the vertex $i$ and denote by $P_i$ and $I_i$ the projective cover and injective envelope of $S_i$, respectively. For a finite dimensional algebra $\Lambda$,  we denote by $D^b(\Lambda)$ its bounded derived category.

\section{Preliminaries}
Let us recall some notations and properties of extriangulated categories and filtration subcategories from \cite{Na} and \cite{Wa}, respectively.

\subsection{Extriangulated categories}
An extriangulated category is
a datum consisting of an
additive category $\mathscr{C}$, a bifunctor $\mathbb{E}:\mathscr{C}^{\rm op}\times \mathscr{C}\rightarrow {\bf Ab}$ and an additive realization $\mathfrak{s}$
that essentially defines the class of conflations, which satisfies certain axioms
that simultaneously generalize axioms of exact categories and triangulated categories, denoted by (ET1)-(ET4), (ET3)$^{\rm op}$ and (ET4)$^{\rm op}$ (see \cite[Definition 2.12]{Na}).
Exact categories, triangulated categories and extension-closed subcategories of triangulated categories are
extriangulated categories.

Let $\mathfrak{s}$ be an additive realization of $\mathbb{E}$. If $\mathfrak{s}(\delta)=[A\stackrel{x}{\longrightarrow}B\stackrel{y}{\longrightarrow}C]$, then the sequence $A\stackrel{x}{\longrightarrow}B\stackrel{y}{\longrightarrow}C$ is called a {\em conflation}, $x$ is called an {\em inflation} and $y$ is called a {\em deflation}. In this case, we say that $A\stackrel{x}{\longrightarrow}B\stackrel{y}{\longrightarrow}C\stackrel{\delta}\dashrightarrow$ is an $\mathbb{E}$-triangle. We will write $A=\cocone(y)$ and $C=\cone(x)$ if necessary.  For more details, we refer the reader to \cite[Section 2]{Na}.

In what follows, we always assume $\mathscr{C}$ is an extriangulated category.
\subsection{ Filtration subcategories}
Let $\mathcal{X}$ be a class of objects in $\mathscr{C}$. The {\em filtration subcategory} $\mathbf{Filt_{\mathscr{C}}(\mathcal{X})}$ is consisting of all objects $M$ admitting a finite filtration of the form
\begin{equation*}
0=M_{0}\stackrel{f_{0}}{\longrightarrow}M_{1}\stackrel{f_{1}}{\longrightarrow}M_{2}{\longrightarrow}\cdots\stackrel{f_{n-1}}{\longrightarrow}M_{n}=M
\end{equation*}
with $f_{i}$ being an inflation and $\cone(f_{i})\in\mathcal{X}$ for any $0\leq i\leq n-1$. In this case, we say that $M$ possesses an $\mathcal{X}$-{\em filtration} of length $n$ and write $M\sim(\cone(f_{0}),\cdots,\cone(f_{n-1}))_{\mathcal{X}}$.
For each object $M\in\Filt_{\mathscr{C}}(\mathcal{X})$, the minimal length of $\mathcal{X}$-{filtrations} of $M$ is called the $\mathcal{X}$-{\em length} of $M$, which is denoted by $l_{\mathcal{X}}(M)$.

An object in an additive category $\mathcal{C}$ is called a {\em brick}, if its endomorphism ring is a division ring. A set $\mathcal{X}$ of isoclasses of bricks in $\mathcal{C}$ is called a {\em semibrick} if $\Hom_{\mathcal{C}}(X_1,X_2)=0$ for any two non-isomorphic objects $X_1,X_2$ in $\mathcal{X}$. The following result will be frequently used in what follows, see \cite{Wa}, Lemmas 2.8, 2.9, 3.5, 5.4.
\begin{proposition}\label{2-1} Let $\mathcal{X}$ be a semibrick in $\mathscr{C}$.

$(1)$ $\mathbf{Filt_{\mathscr{C}}(\mathcal{X})}$ is the smallest extension-closed subcategory containing $\mathcal{X}$ in $\mathscr{C}$.

$(2)$ For each object $M\in\mathbf{Filt_{\mathscr{C}}(\mathcal{X})}$, there exists two $\mathbb{E}$-triangles $X_{1}\stackrel{}{\longrightarrow}M\stackrel{}{\longrightarrow}M'\stackrel{}\dashrightarrow$ and $M''\stackrel{}{\longrightarrow}M\stackrel{}{\longrightarrow}X_{2}\stackrel{}\dashrightarrow$ with $X_{1},X_{2}\in\mathcal{X}$ and $l_{\mathcal{X}}(M')=l_{\mathcal{X}}(M'')=l_{\mathcal{X}}(M)-1$.

$(3)$ If $f:X\rightarrow M$ is a non-zero morphism in $\mathbf{Filt_{\mathscr{C}}(\mathcal{X})}$ with $X\in\mathcal{X}$, then $f$ is an inflation and $l_{\mathcal{X}}(\cone(f))=l_{\mathcal{X}}(M)-1$.

$(4)$ $\mathbf{Filt_{\mathscr{C}}(\mathcal{X})}$ is closed under direct summands.

$(5)$ If $A=B\oplus C\in\mathbf{Filt_{\mathscr{C}}(\mathcal{X})}$, then $l_{\mathcal{X}}(A)=l_{\mathcal{X}}(B)+l_{\mathcal{X}}(C)$.
\end{proposition}

\section{Jordan-H\"{o}lder property}
We denote by $[X_{1},\cdots, X_{n}]$ the sequence of the objects $X_{1},\cdots, X_n$ in $\mathscr{C}$. For any two sequences $[X_{1},\cdots, X_{n}]$ and $[Y_{1},\cdots, Y_{m}]$ of objects in $\mathscr{C}$, they are said to be {\em isomorphic}, denoted by $[X_{1},\cdots, X_{n}]\cong[Y_{1},\cdots, Y_{m}]$, if $n=m$ and there exists a permutation $\sigma$ of the set $\{1,2,\cdots,n\}$ such that $X_{i}\cong Y_{\sigma(i)}$ for any $1\leq i\leq n$.

\begin{definition}\label{cs}
We say that the extriangulated category $\mathscr{C}$ satisfies the {\em weak Jordan-H\"{o}lder property}, abbreviated by (WJHP), if the following conditions hold:

$(1)$ There exists a set $\mathcal{X}$ of objects in $\mathscr{C}$ such that $\mathscr{C}=\mathbf{Filt_{\mathscr{C}}(\mathcal{X})}$.

$(2)$ For each object $M\in\mathscr{C}$ with $l_{\mathcal{X}}(M)=n$, all $\mathcal{X}$-filtrations of $M$ with length $n$ are {\em isomorphic} to each other. That is, if $M\sim(X_{1},\cdots, X_{n})_{\mathcal{X}}$ and $M\sim(Y_{1},\cdots, Y_{n})_{\mathcal{X}}$, then $[X_{1},\cdots, X_{n}]\cong[Y_{1},\cdots, Y_{n}]$.
In addition, if for any object $M\in\mathscr{C}$, all $\mathcal{X}$-filtrations of $M$ have the same length, then we say that $\mathscr{C}$ satisfies the {\em Jordan-H\"{o}lder property}, abbreviated by (JHP).
\end{definition}

As in the module categories, we can define simple objects, composition series in extriangulated categories. A non-zero object $M$ in $\mathscr{C}$ is called a {\em simple} object if there does not exist an $\mathbb{E}$-triangle $A\stackrel{}{\longrightarrow}M\stackrel{}{\longrightarrow}B\stackrel{}\dashrightarrow$ in $\mathscr{C}$ such that $A,B\neq0$. Let $\cim(\mathscr{C})$ be the set of isoclasses of simple objects in $\mathscr {C}$. However, $\cim(\mathscr {C})$ could be an empty set, i.e., there exist no simple objects in $\mathscr{C}$.

Assume that $\cim(\mathscr{C})\neq\emptyset$.
Given an object $M$ in $\mathscr{C}$, a finite filtration
\begin{equation*}
0=M_{0}\stackrel{f_{0}}{\longrightarrow}M_{1}\stackrel{f_{1}}{\longrightarrow}M_{2}{\longrightarrow}\cdots\stackrel{f_{n-1}}{\longrightarrow}M_{n}=M
\end{equation*}
is called a {\em composition series} of $M$ if each $f_{i}$ is an inflation and $\cone(f_{i})\in\cim(\mathscr{C})$ for any $0\leq i\leq n-1$.
\begin{proposition}\label{JHP}
Let $\mathscr{C}$ be an extriangulated category. Then $\mathscr{C}$ satisfies ${\rm (JHP)}$ if and only if the following conditions hold.

$(1)$ Every object in $\mathscr{C}$ has at least one composition series. That is, $\mathscr{C}=\mathbf{Filt_{\mathscr{C}}(\cim(\mathscr{C})})$.

$(2)$ For each object $M\in\mathscr{C}$,  all composition series of $M$ are isomorphic to each other.
\end{proposition}
\begin{proof}
Clearly, we just need to prove the ``only if" part. Suppose that $\mathscr{C}$ satisfies ${\rm (JHP)}$.
By Definition \ref{cs}, there exists a set $\mathcal{X}$ of objects in $\mathscr{C}$ such that $\mathscr{C}=\mathbf{Filt_{\mathscr{C}}(\mathcal{X})}$.
Now, we need to prove that $\mathcal{X}=\cim(\mathscr{C})$. For any $\mathbb{E}$-triangle $A\stackrel{}{\longrightarrow}B\stackrel{}{\longrightarrow}C\stackrel{}\dashrightarrow$,  there exists an $\mathcal{X}$-filtration  of $B$ with length $l_{\mathcal{X}}(A)+l_{\mathcal{X}}(C)$ (see the proof of \cite[Lemma 2.7]{Wa}). Since $\mathscr{C}$ satisfies {\rm (JHP)}, all $\mathcal{X}$-filtrations of $B$ have the same length, we get that $l_{\mathcal{X}}(B)=l_{\mathcal{X}}(A)+l_{\mathcal{X}}(C)$. So for any object $X\in\mathcal{X}$, suppose that there exists an $\mathbb{E}$-triangle $A\stackrel{}{\longrightarrow}X\stackrel{}{\longrightarrow}C\stackrel{}\dashrightarrow$, then $l_{\mathcal{X}}(X)=l_{\mathcal{X}}(A)+l_{\mathcal{X}}(C)$. While $l_{\mathcal{X}}(X)=1$, thus, either $A=0$ or $C=0$, i.e., $X\in\cim(\mathscr{C})$. On the other hand, $\cim(\mathscr {C})\subseteq\mathscr{C}=\mathbf{Filt_{\mathscr{C}}(\mathcal{X})}$, by the definition of simple objects, we get that $\cim(\mathscr {C})\subseteq\mathcal{X}$. Therefore, $\mathcal{X}=\cim(\mathscr{C})$.
\end{proof}

\begin{lemma}\label{length} Let $\mathcal{X}$ be a class of objects in $\mathscr{C}$, $M\in\Filt_{\mathscr{C}}(\mathcal{X})$ with $l_{\mathcal{X}}(M)=n$. Given an $\mathcal{X}$-filtration $$0=M_{0}\stackrel{f_{0}}{\longrightarrow}M_{1}\stackrel{f_{1}}{\longrightarrow}M_{2}\stackrel{f_{2}}{\longrightarrow}\cdots\stackrel{f_{n-1}}{\longrightarrow}M_{n}=M.$$

$(1)$ $l_{\mathcal{X}}(M_{i})=i$ for $0\leq i\leq n$.

$(2)$ $l_{\mathcal{X}}(\cone(f_{j}f_{j-1}\cdots f_{i}))=j-i+1$ for $0\leq i\leq j\leq n-1$.

$(3)$ Set $f=f_{n-1}\cdots f_{1}$. Then $l_{\mathcal{X}}(\cone(f))=n-1$ and $\cone(f)\sim(\cone(f_{1}),\cdots, \cone(f_{n-1}))_{\mathcal{X}}$.
\end{lemma}
\begin{proof} The first two statements follow from  \cite[Lemma 2.9]{Wa}. In particular, $l_{\mathcal{X}}(\cone(f))=n-1$. For any $2\leq i-1\leq n-1$, the axiom (ET4) of extriangulated categories yields the following commutative diagram
$$\xymatrix{
 M_{1} \ar@{=}[d] \ar[r]^-{f_{i-2}\cdots f_{1}}& M_{i-1}\ar[d]^{f_{i-1}} \ar[r] & \cone(f_{i-2}\cdots f_{1}) \ar[d]^{l_{i-2}}\\
  M_{1}  \ar[r] & M_{i}\ar[d] \ar[r] & \cone(f_{i-1}\cdots f_{1}) \ar[d]\\
  &  \cone(f_{i-1})\ar@{=}[r] & \cone(f_{i-1}). }
$$
Hence, we obtain an $\mathcal{X}$-filtration
\begin{equation*}
0\stackrel{l_{0}}{\longrightarrow}\cone(f_{1})\stackrel{l_{1}}{\longrightarrow}\cone(f_{2}f_{1})\stackrel{l_{2}}{\longrightarrow}\cdots\stackrel{l_{n-2}}{\longrightarrow}\cone(f)
.\end{equation*}
Therefore,  $\cone(f)\sim(\cone(f_{1}),\cdots, \cone(f_{n-1}))_{\mathcal{X}}$.
\end{proof}
Let us state the first main
result in this paper as the following
\begin{theorem}\label{main}
Let $\mathcal{X}$ be a semibrick and $\mathcal{T}=\Filt_{\mathscr{C}}(\mathcal{X})$. Then $\mathcal{T}$ satisfies {\rm (WJHP)}.
\end{theorem}
\begin{proof} Take any object $M\in\mathcal{T}$ with $l_{\mathcal{X}}(M)=n$. Let
\begin{equation*}
0=X_{0}\stackrel{f_{0}}{\longrightarrow}X_{1}\stackrel{f_{1}}{\longrightarrow}X_{2}{\longrightarrow}\cdots\stackrel{f_{n-1}}{\longrightarrow}X_{n}=M
\end{equation*}
and
\begin{equation*}
0=Y_{0}\stackrel{g_{0}}{\longrightarrow}Y_{1}\stackrel{g_{1}}{\longrightarrow}Y_{2}{\longrightarrow}\cdots\stackrel{g_{n-1}}{\longrightarrow}Y_{n}=M
\end{equation*}
be two $\mathcal{X}$-filtrations of $M$. In what follows, we prove that
$$[\cone(f_{0}),\cone(f_{1}),\cdots,\cone(f_{n-1})]\cong[\cone(g_{0}),\cone(g_{1}),\cdots,\cone(g_{n-1})].$$

The cases of $n=0,1$ are trivial. Set $f=f_{n-1}\cdots f_{1}$. Consider the following diagram
\begin{equation}\label{three}
\xymatrix{
   X_{1} \ar[r]^-{f} & M \ar@{=}[d] \ar[r]^-{h_{1}} & \cone(f) \\
    Y_{n-1}\ar[r]^-{g_{n-1}} &M  \ar[r]^-{h_{2}} & \cone(g_{n-1}),  }
\end{equation}
where $\cone(g_{n-1}){\in}\mathcal{X}$ and $l_{\mathcal{X}}(\cone(f))=n-1$. Note that $h_{2}f$ is zero or an isomorphism.

First of all, let us prove for the case of $n=2$. If $h_{2}f=h_{2}f_{1}$ is an isomorphism, i.e., there exists a morphism $x:\cone(g_{1})\rightarrow X_1$ such that $xh_2f={\rm id}_{X_1}$ and $h_2fx={\rm id}_{{\rm cone}(g_{1})}$, then $f$ is a splitting monomorphism and $h_2$ is a splitting epimorphism. Thus, $M\cong X_{1}\oplus\cone(f_{1})\cong Y_{1}\oplus\cone(g_{1})$ and then $[\cone(f_{0}),\cone(f_{1})]\cong[\cone(g_{0}),{\rm cone}(g_{1})]$. If $h_{2}f_{1}=0$, there exist $a: X_{1}\rightarrow Y_{1}$ and $b:\cone(f)\rightarrow\cone(g_{1})$ such that (\ref{three}) is commutative. Note that $b$ is either zero or an isomorphism. If $b=0$, then $h_{2}=0$ and $M$ is isomorphic to a direct summand of $Y_{1}$. So we have that $l_{\mathcal{X}}(M)\leq l_{\mathcal{X}}(Y_{1})=1$ by Proposition \ref{2-1}(5). This is a contradiction. Thus, $b$ is an isomorphism and then so is $a$. Hence, $[\cone(f_{0}),\cone(f_{1})]\cong[\cone(g_{0}),\cone(g_{1})]$.

Now we consider the case of $n\geq2$.

$\mathbf{Case}$ 1  If $h_{2}f$ is an isomorphism, then $M\cong X_{1}\oplus\cone(f)\cong Y_{n-1}\oplus\cone(g_{n-1})$. Note that $l_{\mathcal{X}}(Y_{n-1})=n-1$ and $Y_{n-1}\sim(\cone(g_{0}),\cone(g_{1}),\cdots,\cone(g_{n-2}))_{\mathcal{X}}$. By Lemma \ref{length}, $l_{\mathcal{X}}(\cone(f))=n-1$ and $\cone(f)\sim(\cone(f_{1}),\cone(f_{2}),\cdots,\cone(f_{n-1}))_{\mathcal{X}}$.

If $\cone(f)\cong Y_{n-1}$, by induction,
$$[\cone(f_{1}),\cone(f_{2}),\cdots,\cone(f_{n-1})]\cong[\cone(g_{0}),\cone(g_{1}),\cdots,\cone(g_{n-2})].$$
In this case, $\cone(f_{0})=X_{1}\cong\cone(g_{n-1})$, we get that
$$[\cone(f_{0}),\cone(f_{1}),\cdots,\cone(f_{n-1})]\cong[\cone(g_{0}),\cone(g_{1}),\cdots,\cone(g_{n-1})].$$

If $\cone(f)\ncong Y_{n-1}$, then there exists a non-trivial decomposition $Y_{n-1}=L_{1}\oplus L_{2}$ such that $\cone(f)\cong L_{1}\oplus\cone(g_{n-1})$ and $X_{1}\cong L_{2}$. Here
$l_{\mathcal{X}}(L_{1})=l_{\mathcal{X}}(Y_{n-1})-l_{\mathcal{X}}(L_{2})=n-2$. Assume that $L_{1}\sim(N_{0},N_{1},\cdots, N_{n-3})_{\mathcal{X}}$. Since $L_{1}\stackrel{}{\longrightarrow}Y_{n-1}\stackrel{}{\longrightarrow}L_{2}\stackrel{0}\dashrightarrow$ is an $\mathbb{E}$-triangle, we have that $Y_{n-1}\sim(N_{0},N_{1},\cdots, N_{n-3},L_{2})_{\mathcal{X}}$. By induction, $[N_{0},N_{1},\cdots, N_{n-3},L_{2}]\cong[\cone(g_{0}),\cone(g_{1}),\cdots,\cone(g_{n-2})]$.
Similarly, we have that \begin{equation}\label{nage}[N_{0},N_{1},\cdots, N_{n-3},\cone(g_{n-1})]\cong[\cone(f_{1}),\cone(f_{2}),\cdots,\cone(f_{n-1})],\end{equation}
since $L_{1}\stackrel{}{\longrightarrow}\cone(f)\stackrel{}{\longrightarrow}\cone(g_{n-1})\stackrel{0}\dashrightarrow$ is an $\mathbb{E}$-triangle.
By (\ref{nage}), $\cone(g_{n-1})\cong\cone(f_{i})$ for some $1\leq i\leq n-1$ and $$[N_{0},N_{1},\cdots, N_{n-3}]\cong[\cone(f_{1}),\cdots,\cone(f_{i-1}),\cone(f_{i+1}),\cdots,\cone(f_{n-1})].$$
Thus, we get that
\begin{flalign*}
&[\cone(g_{0}),\cone(g_{1}),\cdots,\cone(g_{n-2})]\\&\cong[N_{0},N_{1},\cdots, N_{n-3},L_{2}]
\cong[N_{0},N_{1},\cdots, N_{n-3},\cone(f_{0})]\\
&\cong[\cone(f_{0}),\cone(f_{1}),\cdots,\cone(f_{i-1}),\cone(f_{i+1}),\cdots,\cone(f_{n-1})].
\end{flalign*}
Since $\cone(g_{n-1})\cong\cone(f_{i})$, we obtain that
$$[\cone(g_{0}),\cone(g_{1}),\cdots,\cone(g_{n-1})]\cong[\cone(f_{0}),\cone(f_{1}),\cdots,\cone(f_{n-1})].$$

$\mathbf{Case}$ 2 If $h_{2}f=0$, there is a commutative diagram
\begin{equation*}
\xymatrix{
   X_{1} \ar[r]^-{f} \ar[d]^-{a} & M \ar@{=}[d] \ar[r]^-{h_{1}} & \cone(f) \ar[d]^-{} \\
    Y_{n-1}\ar[r]^-{g_{n-1}} &M  \ar[r]^-{h_{2}} & \cone(g_{n-1}).  }
\end{equation*}
Since $l_{\mathcal{X}}(\cone(f))=n-1<l_{\mathcal{X}}(M)$, we obtain that $f\neq0$ and thus $a\neq0$. By Proposition \ref{2-1}(3), $a$ is an inflation and $l_{\mathcal{X}}(\cone(a))=l_{\mathcal{X}}(Y_{n-1})-1=n-2$. Let
\begin{equation*}
0=Z_{0}\stackrel{h_{0}}{\longrightarrow}Z_{1}\stackrel{h_{1}}{\longrightarrow}Z_{2}{\longrightarrow}\cdots\stackrel{h_{n-3}}{\longrightarrow}Z_{n-2}=\cone(a)
\end{equation*}
be an $\mathcal{X}$-filtration of $\cone(a)$. For the inflation $h_{n-3}:Z_{n-3}\rightarrow Z_{n-2}$, ${\rm (ET4)^{op}}$ yields the following commutative diagram
$$\xymatrix{
 X_{1} \ar@{=}[d] \ar[r]^-{s_{1}}& N_{n-2}\ar[d]^{t_{n-2}} \ar[r] & Z_{n-3} \ar[d]^{h_{n-3}}\\
  X_{1}  \ar[r]^-{a} & Y_{n-1}\ar[d] \ar[r] &Z_{n-2}  \ar[d]\\
  &  \cone(h_{n-3})\ar@{=}[r] & \cone(h_{n-3}) }
$$
with $l_{\mathcal{X}}(Z_{n-3})=n-3$. Similarly, we can form the following commutative diagram
$$\xymatrix{
 X_{1} \ar@{=}[d] \ar[r]^-{s_{2}}& N_{n-3}\ar[d]^{t_{n-3}} \ar[r] & Z_{n-4} \ar[d]^{h_{n-4}}\\
  X_{1}  \ar[r]^-{s_{1}} & N_{n-2}\ar[d] \ar[r] &Z_{n-3}  \ar[d]\\
  &  \cone(h_{n-4})\ar@{=}[r] & \cone(h_{n-4}) }
$$
with $l_{\mathcal{X}}(Z_{n-4})=n-4$. By repeating this process, we obtain an $\mathcal{X}$-filtration
\begin{equation}\label{1.3}
0=N_{0}\stackrel{}{\longrightarrow}X_{1}\stackrel{t_{1}s_{n-2}}{\longrightarrow}N_{2}\stackrel{t_{2}}{\longrightarrow}\cdots\stackrel{t_{n-4}}{\longrightarrow}N_{n-3}\stackrel{t_{n-3}}{\longrightarrow}N_{n-2}\stackrel{t_{n-2}}{\longrightarrow}Y_{n-1}\stackrel{g_{n-1}}{\longrightarrow}Y_{n}=M.\end{equation}
Recall that $Y_{n-1}\sim(\cone(g_{0}),\cone(g_{1}),\cdots,\cone(g_{n-2}))_{\mathcal{X}}$. By (\ref{1.3}) and induction, we obtain that
\begin{align*}
Y_{n-1}&\sim(X_{1},\cone(t_{1}s_{n-2}),\cone(t_{2}),\cdots,\cone(t_{n-2}))_{\mathcal{X}}\\
&=(X_{1},\cone(h_{0}),\cone(h_{1}),\cdots,\cone(h_{n-3}))_{\mathcal{X}}
\end{align*}
and then $[\cone(g_{0}),\cone(g_{1}),\cdots,\cone(g_{n-2})]\cong[X_{1},\cone(h_{0}),\cone(h_{1}),\cdots,\cone(h_{n-3})]$.
Since $f=g_{n-1}a=g_{n-1}t_{n-2}\cdots t_{1}s_{n-2}$, by Lemma \ref{length}, we have that
\begin{align*}
\cone(f)&\sim(\cone(t_{1}s_{n-2}),\cone(t_{2}),\cdots,\cone(t_{n-2}),\cone(g_{n-1}))_{\mathcal{X}}\\
&=(\cone(h_{0}),\cone(h_{1}),\cdots,\cone(h_{n-3}),\cone(g_{n-1}))_{\mathcal{X}}.
\end{align*}
Recall that $\cone(f)\sim(\cone(f_{1}),\cone(f_{2}),\cdots,\cone(f_{n-1}))_{\mathcal{X}}$, by induction, we obtain that
\begin{equation*}
[\cone(f_{1}),\cone(f_{2}),\cdots,\cone(f_{n-1})]\cong[\cone(h_{0}),\cone(h_{1}),\cdots,\cone(h_{n-3}),\cone(g_{n-1})].
\end{equation*}
Hence, we have that
\begin{align*}
&[\cone(f_{0}),\cone(f_{1}),\cone(f_{2}),\cdots,\cone(f_{n-1})]\\
&\cong[X_{1},\cone(h_{0}),\cone(h_{1}),\cdots,\cone(h_{n-3}),\cone(g_{n-1})]\\
&\cong[\cone(g_{0}),\cone(g_{1}),\cdots,\cone(g_{n-2}),\cone(g_{n-1})].
\end{align*}
Therefore, we complete the proof.
\end{proof}

An extriangulated category $\mathscr{C}$ is called {\em length} if $\mathscr{C}=\Filt_{\mathscr{C}}(\cim(\mathscr{C}))$.
Given an object $M\in\mathscr{C}$, we say that the object $N\in\mathscr{C}$ is a {\em $[1]$-shift} of $M$, if there exists an $\mathbb{E}$-triangle
$M\stackrel{}{\longrightarrow}0\stackrel{}{\longrightarrow}N\stackrel{}\dashrightarrow$. In this case, we write $N$ as $TM$. Clearly, for any $M\in\mathscr{C}$, $TM$ is unique up to isomorphisms if it exists. If $\mathscr{C}$ is an exact category, the $[1]$-shift is well-defined only for the zero object. If $\mathscr{C}$ is a triangulated category, the $[1]$-shift is well-defined for any object. In what follows, when we write $TM\in\mathscr{C}$, it means that $TM$ exists in $\mathscr{C}$.

\begin{lemma}\label{c1} Let
$X\stackrel{f}{\longrightarrow}L\stackrel{g}{\longrightarrow}Y\stackrel{\delta}\dashrightarrow$ be an $\mathbb{E}$-triangle. If $f=0$, then $TX\in\mathscr{C}$ and $Y\cong L\oplus TX$.
\end{lemma}
\begin{proof} Since $f=0$, we obtain that $g$ is a splitting monomorphism. Then there exists a decomposition $Y=Y_{1}\oplus Y_{2}$ such that $ g={g_{1}\choose g_{2}}:L\rightarrow Y_{1}\oplus Y_{2}$ with $g_{1}$ being an isomorphism. By $\rm(ET4)^{op}$, there is a commutative diagram
\begin{equation*}
\xymatrix{
  X \ar@{=}[d]\ar[r]^-{} &H\ar[d]_-{} \ar[r]^-{} & Y_{2}\ar[d]^-{{0\choose 1}}\ar@{-->}[r]& \\
  X \ar[r]^-{0} & L\ar[d]_-{g_{1}} \ar[r]^-{{g_1\choose g_2}} & Y_{1}\oplus Y_{2}\ar[d]^-{(1~0)}\ar@{-->}[r]^-{\delta}& \\
   &Y_{1} \ar@{-->}[d]^{0}\ar@{=}[r] & Y_{1}\ar@{-->}[d]^{0}& \\
&&&}
\end{equation*}
with $H\cong0$, since $g_{1}$ is an isomorphism.  It follows that $TX\cong Y_{2}$ and $Y\cong L\oplus TX$.
\end{proof}

By Theorem \ref{main}, each filtration subcategory generated by a semibrick satisfies (WJHP), but in general, (JHP) does not hold. For this, we introduce the following

\begin{definition} We say a semibrick $\mathcal{X}$ is {\em proper} if $X[1]\notin\Filt_{\mathscr{C}}(\mathcal{X})$ for any $X\in\mathcal{X}$. That is, for any $X\in\mathcal{X}$, there does not exist an $\mathbb{E}$-triangle of the form $X\stackrel{}{\longrightarrow}0\stackrel{}{\longrightarrow}N\stackrel{}\dashrightarrow$ for some $N\in\Filt_{\mathscr{C}}(\mathcal{X})$.
\end{definition}
\begin{definition}$($\cite[Definition 3.1]{Wa}$)$
A set $\mathcal{X}$ of isoclasses of objects in $\mathscr{C}$ is said to be {\em simple} if $\mathcal{X}= \cim(\Filt_{\mathscr{C}}(\mathcal{X}))$.
\end{definition}

\begin{remark}\label{2.2} If $\mathscr{C}$ is an exact category, it is easy to see that $\mathcal{X}$ is a semibrick if and only if $\mathcal{X}$ is a proper semibrick if and only if $\mathcal{X}$ is a simple semibrick. But for triangulated categories, these equivalences do not hold. For more details, see Example \ref{4-2}.
\end{remark}

\begin{lemma}\label{gen} Let $\mathcal{X}$ be a semibrick  in $\mathscr{C}$. The following are equivalent.

$(1)$ $\mathcal{X}$ is a proper semibrick.

$(2)$ If $f:X\rightarrow M$ is an inflation in $\Filt_{\mathscr{C}}(\mathcal{X})$ with $X\in\mathcal{X}$, then $f\neq0$ and $l_{\mathcal{X}}(\cone(f))=l_{\mathcal{X}}(M)-1$.

$(3)$  For any $M\in\Filt_{\mathscr{C}}(\mathcal{X})$, $M[1]\notin\Filt_{\mathscr{C}}(\mathcal{X})$.

$(4)$ If $f:N\rightarrow M$ is an inflation in $\Filt_{\mathscr{C}}(\mathcal{X})$, then $f\neq0$.

$(5)$ For any $\mathbb{E}$-triangle $A\stackrel{}{\longrightarrow}B\stackrel{}{\longrightarrow}C\stackrel{}\dashrightarrow$ in $\Filt_{\mathscr{C}}(\mathcal{X})$, we have that $l_{\mathcal{X}}(B)=l_{\mathcal{X}}(A)+l_{\mathcal{X}}(C)$. In particular, $\mathcal{X}$ is a simple semibrick and thus $\mathcal{X}$-filtrations coincide with composition series in $\Filt_{\mathscr{C}}(\mathcal{X})$.
\end{lemma}
\begin{proof} $(1)\Rightarrow(2)$ Suppose that $f=0$, by Lemma \ref{c1}, $\cone(f)\cong M\oplus X[1]$ and thus $X[1]\in \Filt_{\mathscr{C}}(\mathcal{X})$. This is a contradiction. Hence, $f\neq0$. By Proposition \ref{2-1}(3), $l_{\mathcal{X}}(\cone(f))=l_{\mathcal{X}}(M)-1$. Similarly, one can prove that $(3)\Rightarrow(4)$
and $(2)\Rightarrow(1)$ is clear.

$(1)\Rightarrow(3)$ Suppose that there is an object $M\in\Filt_{\mathscr{C}}(\mathcal{X})$ such that $M[1]\in\Filt_{\mathscr{C}}(\mathcal{X})$. By Proposition \ref{2-1}(2), there exists an  $\mathbb{E}$-triangle $X\stackrel{f}{\longrightarrow}M\stackrel{}{\longrightarrow}M'\stackrel{}\dashrightarrow$ in $\Filt_{\mathscr{C}}(\mathcal{X})$ with $X\in\mathcal{X}$. By ${\rm (ET4)}$, there is an $\mathbb{E}$-triangle $M'\stackrel{f}{\longrightarrow}X[1]\stackrel{}{\longrightarrow}M[1]\stackrel{}\dashrightarrow$ and thus $X[1]\in \Filt_{\mathscr{C}}(\mathcal{X})$. This is a contradiction. Clearly, $(3)\Rightarrow(2)$ and $(4)\Rightarrow(2)$. Therefore, we have obtained the equivalences of the first four statements.

$(1)\Rightarrow(5)$ Take any $\mathbb{E}$-triangle $A\stackrel{f}{\longrightarrow}B\stackrel{}{\longrightarrow}C\stackrel{}\dashrightarrow$ in $\Filt_{\mathscr{C}}(\mathcal{X})$,
we proceed the proof by induction on $l_{\mathcal{X}}(A)=n$. If $n=0$, then $B\cong C$ and thus $l_{\mathcal{X}}(B)=l_{\mathcal{X}}(C)$. If $n=1$, by (2), $l_{\mathcal{X}}(C)=l_{\mathcal{X}}(B)-1$. For $n>1$, by Proposition \ref{2-1}(2), there exists an $\mathbb{E}$-triangle $X\stackrel{}{\longrightarrow}A\stackrel{}{\longrightarrow}A'\stackrel{}\dashrightarrow$ such that $l_{\mathcal{X}}(X)=1$ and $l_{\mathcal{X}}(A')=n-1$. Hence, $\rm (ET4)$ yields the following commutative diagram of conflations
$$\xymatrix{
 X \ar@{=}[d] \ar[r] & A\ar[d]^{} \ar[r] & A' \ar[d]\\
  X  \ar[r] & B \ar[d] \ar[r] & E \ar[d]\\
  & C \ar@{=}[r] & C }
$$
in $\Filt_{\mathscr{C}}(\mathcal{X})$. By induction, we have that $$l_{\mathcal{X}}(B)=l_{\mathcal{X}}(X)+l_{\mathcal{X}}(E)=l_{\mathcal{X}}(X)+l_{\mathcal{X}}(A')+l_{\mathcal{X}}(C)=l_{\mathcal{X}}(A)+l_{\mathcal{X}}(C).$$

$(5)\Rightarrow(2)$ Suppose that $f=0$, by Lemma \ref{c1}, $M$ is a direct summand of $\cone(f)$. It follows that
$$1+l_{\mathcal{X}}(\cone(f))=l_{\mathcal{X}}(M)\leq l_{\mathcal{X}}(\cone(f)).$$
This is a contradiction. Hence, $f\neq 0$.
\end{proof}

\begin{theorem}\label{main2}
Let $\mathcal{X}$ be a semibrick and set $\mathcal{T}=\Filt_{\mathscr{C}}(\mathcal{X})$. Then
$\mathcal{T}$ satisfies {\rm (JHP)} if and only if $\mathcal{X}$ is a proper semibrick.
\end{theorem}
\begin{proof} Assume that $\mathcal{T}$ satisfies {\rm (JHP)}. As mentioned in the proof of Proposition \ref{JHP}, for any $\mathbb{E}$-triangle $A\stackrel{}{\longrightarrow}B\stackrel{}{\longrightarrow}C\stackrel{}\dashrightarrow$ in $\mathcal{T}$, there exists an $\mathcal{X}$-filtration of $B$ with length $l_{\mathcal{X}}(A)+l_{\mathcal{X}}(C)$. Since $\mathcal{T}$ satisfies {\rm (JHP)}, we obtain that $l_{\mathcal{X}}(B)=l_{\mathcal{X}}(A)+l_{\mathcal{X}}(C)$. By Lemma \ref{gen}(5), we get that $\mathcal{X}$ is proper.

Assume that $\mathcal{X}$ is a proper semibrick. Let $M\in\mathcal{T}$. By Theorem \ref{main}, we only need to prove all $\mathcal{X}$-filtrations of $M$ have the same length. Indeed, for any $\mathcal{X}$-filtration
\begin{equation*}
0=M_{0}\stackrel{f_{0}}{\longrightarrow}M_{1}\stackrel{f_{1}}{\longrightarrow}M_{2}\stackrel{f_{2}}{\longrightarrow}\cdots\stackrel{f_{n-1}}{\longrightarrow}M_{n}=M,
\end{equation*}
by Lemma \ref{gen}(5), we have that
$$l_{\mathcal{X}}(M)=\sum_{i=0}^{n-1}l_{\mathcal{X}}(\cone(f_{i}))=n.$$
That is, every $\mathcal{X}$-filtration of $M$ is of length $l_{\mathcal{X}}(M)$.
\end{proof}

\begin{corollary} An abelian category $\mathcal{A}$ satisfies {\rm (JHP)} if and only if $\mathcal{A}$ is a length abelian category.
\end{corollary}
\begin{proof} If $\mathcal{A}$ satisfies {\rm (JHP)}, by Proposition \ref{JHP}, we have that $\Filt(\cim(\mathcal{A}))=\mathcal{A}$, i.e., $\mathcal{A}$ is a length abelian category. Conversely, if $\mathcal{A}$ is a length abelian category, i.e., $\Filt(\cim(\mathcal{A}))=\mathcal{A}$, noting that $\cim(\mathcal{A})$ is a proper semibrick in $\mathcal{A}$, we obtain that $\mathcal{A}$ satisfies {\rm (JHP)} by Theorem \ref{main2}.
\end{proof}

\begin{corollary} Let $\mathcal{X}$ be a semibrick in an exact category $\mathcal {E}$ and $\mathcal{T}=\Filt_{\mathcal {E}}(\mathcal{X})$. Then $\mathcal{T}$ satisfies {\rm (JHP)}.
\end{corollary}
\begin{proof}
We only need to note that each semibrick in an exact category is proper. Then by Theorem \ref{main2}, we complete the proof.
\end{proof}

Let $Q$ be a finite acyclic quiver. A {\em torsion free class} in $\mod kQ$ is a full subcategory which is closed under extensions and submodules. We remark that the reverse of Theorem \ref{main} may not hold, i.e., there exists a set $\mathcal{X}$ which is not a semibrick such that $\Filt_{\mathscr{C}}(\mathcal{X})$ satisfies {\rm (WJHP)}.

\begin{example}
Let $Q$ be the quiver $1\longrightarrow2\longleftarrow3$. The Auslander-Reiten quiver of $\mod kQ$ is as follows:
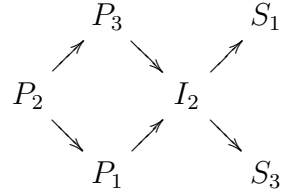
\begin{equation*}
\xymatrix@!=0.5pc{
   &  P_3\ar[dr]^{} & & S_1  \\
 P_2 \ar[dr]\ar[ur]  & &  I_2\ar[dr]^{}\ar[ur]^{}&     \\
  &  P_1\ar[ur]^{} & &  S_3  }
  \end{equation*}
Set $\mathcal{T}=\add\{P_1,P_2,P_3,I_2,S_3\}$. Then $\cim(\mathcal{T})=\{P_1,P_2,S_3\}$. By \cite[Corollary 5.19]{En2}, $\mathcal{T}$ is a torsion-free class satisfying {\rm (JHP)}. Since $\Hom_{{\rm \textbf{mod}}kQ}(P_2,P_1)\neq0$, $\cim(\mathcal{T})$ is not a semibrick.
\end{example}

As we have said in Remark \ref{2.2}, the notions of semibricks, proper semibricks and simple semibricks are coincident in exact categories. However, they are different in  triangulated categories.
\begin{example}\label{4-2}
Let $A$ be the path algebra of the quiver $1\longrightarrow2\longrightarrow3$. The Auslander-Reiten quiver $\Gamma$ of the bounded derived category $D^{b}(A)$ is as follows:
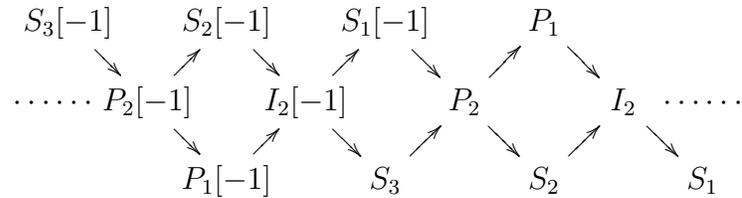
\begin{equation*}
\xymatrix@!=0.5pc{
   && S_3[-1]\ar[dr]  && S_2[-1]\ar[dr] && S_1[-1]\ar[dr] && P_1\ar[dr] && \\
   &&\cdots\cdots\quad& P_2[-1]\ar[dr]\ar[ur] && I_2[-1]\ar[ur]\ar[dr] && P_2\ar[ur]\ar[dr] && I_2\ar[dr] & \cdots\cdots\\
   &&&& P_1[-1]\ar[ur] && S_3\ar[ur] && S_2\ar[ur] && S_1}
\end{equation*}

$(1)$ Clearly, the set $\mathcal{X}_0$ consisting of the isoclasses of objects in the top row of $\Gamma$ is a semibrick in $D^{b}(A)$ and $\Filt_{D^{b}(A)}(\mathcal{X}_0)=D^{b}(A)$.  By Theorem \ref{main}, $D^{b}(A)$ satisfies {\rm (WJHP)}.

$(2)$ Let $\mathcal{X}_1$ be the set consisting of the isoclasses of objects in
$\{S_3[-1],S_2[-1],S_1[-1],P_{1}\}$, and it is a semibrick in $D^{b}(A)$.
Noting that $S_{3}\in\Filt_{D^{b}(A)}(\mathcal{X}_1)$, we obtain that $\mathcal{X}_1$ is not proper. Since $P_{2}[-1]\stackrel{}{\longrightarrow}S_2[-1]\stackrel{}{\longrightarrow}S_{3}\stackrel{}{\longrightarrow} P_{2}$ is a triangle in $\Filt_{D^{b}(A)}(\mathcal{X}_1)$, we get that $S_2[-1]$ is not simple in $\Filt_{D^{b}(A)}(\mathcal{X}_1)$ and thus $\mathcal{X}_1$ is not simple. By Theorem \ref{main2}, {\rm (JHP)} fails in $\Filt_{D^{b}(A)}(\mathcal{X}_1)$.

$(3)$ Let $\mathcal{X}_2$ be the set consisting of the isoclasses of objects in
$\{S_2[-1],S_1[-1],P_{1}\}$.
Clearly, $\mathcal{X}_2$ is a proper semibrick. By Theorem \ref{main2}, $\Filt_{D^{b}(A)}(\mathcal{X}_2)$ satisfies {\rm (JHP)}.

$(4)$ Let $\mathcal{X}_3$ be the set consisting of the isoclasses of objects in
$\{S_2[-1],S_1[-1],S_1\}$.
Clearly, $\mathcal{X}_3$ is a simple semibrick, but not proper.
\end{example}

\begin{example}
Let $\Lambda$ be the path algebra of the quiver $n\longrightarrow n-1\longrightarrow\cdots\longrightarrow 1$.

$(1)$ Let $X,Y\in\mod \Lambda$ be indecomposable.
\begin{itemize}
\item [(i)] If $j-i=0$, then $\Hom_{ D^{b}(\Lambda)}(X[i],Y[j])\cong\Hom_{\Lambda}(X,Y)$ and $\Hom_{ D^{b}(\Lambda)}(Y[j],X[i])\cong\Hom_{\Lambda}(Y,X)$.
\item [(ii)] If $j-i=1$, then  $\Hom_{ D^{b}(\Lambda)}(X[i],Y[j])\cong\Ext^{1}_{\Lambda}(X,Y)$ and $\Hom_{ D^{b}(\Lambda)}(Y[j],X[i])=0$.
\item [(iii)] If $j-i=-1$, then $\Hom_{ D^{b}(\Lambda)}(X[i],Y[j])=0$ and $\Hom_{ D^{b}(\Lambda)}(Y[j],X[i])\cong\Ext^{1}_{\Lambda}(Y,X)$.
\item [(iv)] If $|j-i|\geq2$, then $\Hom_{ D^{b}(\Lambda)}(X[i],Y[j])=\Hom_{ D^{b}(\Lambda)}(Y[j],X[i])=0$.
\end{itemize}

Hence, $\{X[i],Y[j]\}$ is a semibrick in $D^{b}(\Lambda)$ if and only if one of the following conditions holds:
\begin{itemize}
\item [(i)] If $j-i=0$, then $\{X,Y\}$ is a semibrick.
\item [(ii)] If $j-i=1$, then  $\Ext^{1}_{\Lambda}(X,Y)=0$.
\item [(iii)] If $j-i=-1$, then $\Ext^{1}_{\Lambda}(Y,X)=0$.
\item [(iv)] $|j-i|\geq2$.
\end{itemize}

$(2)$ For $0\leq i< j\leq n$, we denote by $M_{ij}$ the indecomposable $\Lambda$-module whose  representation $(M_{i},\varphi_{i})$ is given by
\begin{equation*}0\longrightarrow\cdots\longrightarrow0\longrightarrow k\stackrel{1}\longrightarrow\cdots \stackrel{1}\longrightarrow k\longrightarrow0\longrightarrow\cdots \longrightarrow 0.\end{equation*}
That is, $M_{l}=k$ if $i< l\leq j$ and $M_{l}=0$ otherwise.

Given an indecomposable $\Lambda$-module $X$, we set
\begin{flalign*}&H_{+}^{0}(X)=\{Y\in\ind(\mod \Lambda)~|~\Hom_{\Lambda}(X,Y)\neq0\},\\
&H_{-}^{0}(X)=\{Y\in\ind(\mod \Lambda)~|~\Hom_{\Lambda}(Y,X)\neq0\},\\
&H_{+}^{1}(X)=\{Y\in\ind(\mod \Lambda)~|~\Ext^{1}_{\Lambda}(X,Y)\neq0\},\\
&H_{-}^{1}(X)=\{Y\in\ind(\mod \Lambda)~|~\Ext^{1}_{\Lambda}(Y,X)\neq0\}.\end{flalign*}
Indeed, by \cite[Lemma 3.1]{Ar}, we have that
\begin{flalign*}&H_{+}^{0}(M_{ij})=\{M_{st}~|~i\leq s\leq j-1,j\leq t\leq n\},\\
&H_{-}^{0}(M_{ij})=\{M_{st}~|~0\leq s\leq i,i+1\leq t\leq j\},\\
&H_{+}^{1}(M_{ij})=\{M_{st}~|~0\leq s\leq i-1,i\leq t\leq j-1\},\\
&H_{-}^{1}(M_{ij})=\{M_{st}~|~i+1\leq s\leq j,j+1\leq t\leq n\}.\end{flalign*}

$(3)$ Let $\mathcal{X}$ be a set of isoclasses of objects in $D^{b}(\Lambda)$. By definition,  $\mathcal{X}$ is a semibrick if and only if $\{X,Y\}$ is a semibrick for any $X,Y\in\mathcal{X}$. Set $X=M_{ij}[k]$ and $Y=M_{st}[l]$ for some $M_{ij},M_{st}\in\ind(\mod \Lambda)$ and $k,l\in\mathbb{Z}$. Using $(1)$ and $(2)$ (seeing the diagram in \cite[Lemma 3.1]{Ar}), $\{X,Y\}$ is a semibrick if and only if one of the following conditions holds:
\begin{itemize}
\item [(a)] If $l-k=0$, then $\{M_{ij},M_{st}\}$ is a semibrick in $\mod\Lambda$, i.e., $M_{st}\notin H_{+}^{0}(M_{ij})\cup H_{-}^{0}(M_{ij})$. So, we have that
\begin{itemize}
\item [(i)] If $0\leq s\leq i-1$, then $s+1\leq t\leq i$ or $j+1\leq t\leq n$.
\item [(ii)] If $i+1\leq s\leq j-2$, then  $s+1\leq t\leq j-1$.
\item [(iii)]  If $j\leq s\leq n$, then $s+1\leq t\leq n$.
\end{itemize}
In particular, $s\neq i$ and $s\neq j-1$.
\item [(b)] If $l-k=1$, then  $\Ext^{1}_{\Lambda}(M_{ij},M_{st})=0$, i.e., $M_{st}\notin H_{+}^{1}(M_{ij})$. So, we have that
\begin{itemize}
\item [(i)] If $0\leq s\leq i-1$, then $s+1\leq t\leq i-1$ or $j\leq t\leq n$.
\item [(ii)] If $i\leq s\leq n$, then  $s+1\leq t\leq n$.
\end{itemize}
\item [(c)] If $l-k=-1$, then  $\Ext^{1}_{\Lambda}(M_{st},M_{ij})=0$,  i.e., $M_{st}\notin H_{-}^{1}(M_{ij})$. So, we have that
\item [(i)] If $0\leq s\leq i$, then $s+1\leq t\leq n$.
\item [(ii)] If $i+1\leq s\leq j-1$, then  $s+1\leq t\leq j$.
\item [(iii)] If $j+1\leq s\leq n$, then $s+1\leq t\leq n$.

   \item [(d)]  $|l-k|\geq2$.
\end{itemize}
In particular, by (a), $\mathcal{X}$ is a semibrick in $\mod\Lambda$ if and only if it satisfies the following conditions
\begin{itemize}
\item [(i)] $\mathcal{X}=\{M_{x_{i},y_{i}}\}_{i=1}^{t}$ with $x_{1}<x_{2}<\cdots< x_{t}$ and $t\leq n$.
\item [(ii)] For any $1\leq j<i\leq t$, either $y_{j}\leq x_{i}$ or $y_{j}\geq y_{i}+1$.
\end{itemize}
By \cite[Theorem 4.1]{Asa}, the number of semibricks in $\mod \Lambda$ is equal to the {\em Catalan number}
 $$\frac{1}{n+2}\tiny\begin{pmatrix} 2n+2 \\n+1 \end{pmatrix}.$$
\end{example}

\section{Torsion-free classes satisfying  {\rm (JHP)}}
In this section, we investigate the torsion-free classes satisfying  {\rm (JHP)}  for a quiver of type $A$ by using reflection functors, $c$-sortable elements and Coxeter groups.

\subsection{Reflection functors}
In this subsection, let $Q$ be a finite acyclic quiver.  We say that a vertex $i$ of $Q$ is a {\em sink} if all the arrows incident to $i$ point towards
$i$. Dually, we say that $i$ is a {\em source} if all arrows incident to $i$ point away from it. If $i$ is a sink or a source, we define $\mu_{i}(Q)$ to be the quiver obtained by changing the direction of all the arrows incident to $i$.
If $i$ is a sink of $Q$, then there is a functor $R^{+}_{i}:\rep(Q)\rightarrow\rep(\mu_{i}(Q))$ defined by
$$R^{+}_{i}(M)_{j}=\left\{
\begin{aligned}
~M_{j}, &  & \text{if}~j\neq i, \\
\ker(\bigoplus_{i'\rightarrow i\in Q}M_{i'}\rightarrow M_{i}), &  & \text{if}~j=i.
\end{aligned}
\right.
$$
for any $M\in\rep(Q)$. Similarly, if $i$ is a source of $Q$, we also have a functor $R^{-}_{i}:\rep(Q)\rightarrow\rep(\mu_{i}(Q))$ (cf. \cite{Th}).

Let us collect some properties of $R^{+}_{i}$ and $R^{-}_{i}$, which will be used in the sequel.

\begin{proposition}\label{D11} Let $i$ be a sink of $Q$ and $M\in\rep(Q)$.

$(1)$  If $M\cong S_{i}$, then $R^{+}_{i}(M)=0$. If $M$ has no direct summand isomorphic to $S_{i}$, then
$$\dim (R^{+}_{i}(M))_{j}=\left\{
\begin{aligned}
~\dim M_{j}, &  & \text{if}~j\neq i, \\
(\sum_{i'\rightarrow i\in Q}\dim M_{i'})-\dim M_{i}, &  & \text{if}~j=i.
\end{aligned}
\right.
$$

$(2)$ $R^{+}_{i}$ is left exact and $R^{-}_{i}$ is right exact.

$(3)$ $R_{i}^{-}R_{i}^{+}(M)$ is isomorphic to the direct sum of the indecomposable summands of $M$ which are not isomorphic to $S_{i}$. Similarly, if $i$ is a source of $Q$, then $R_{i}^{+}R_{i}^{-}(M)$ is isomorphic to the direct sum of the indecomposable summands of $M$ which are not isomorphic to $S_{i}$.

$(4)$ Let $\mathcal{F}$ be a subcategory of $\rep(Q)$. If $S_{i}\in \mathcal{F}$, then $\mathcal{F}=\add\{\ind(R_{i}^{-}R_{i}^{+}(\mathcal{F}))\cup\{S_{i}\}\}$. Otherwise, $\mathcal{F}=\add\{\ind(R_{i}^{-}R_{i}^{+}(\mathcal{F}))\}$.

$(5)$ If $\mathcal{F}$ is a torsion-free class in $\rep(Q)$, then $R_{i}^{+}(\mathcal{F})$ is a torsion-free class in $\rep(\mu_{i}(Q))$ which
does not contain $S'_{i}$, where $S'_{i}$ is the i-th simple representation of $\mu_{i}(Q)$.

\end{proposition}
\begin{proof} For (1)--(4), one can see for example \cite[Proposition 3.1]{Th}. Using \cite[Proposition 4.1]{Th} together with \cite[Proposition 4.2]{Th}, we get $(5)$.
\end{proof}

\begin{proposition}\label{MM1}  Let $i$ be a sink of $Q$ and let $\mathcal{F}$ be a torsion-free class in $\rep(Q)$. Set $\mathcal{F}'=R^{+}_{i}(\mathcal{F})$.

$(1)$ The functor $R_{i}^{+}$ gives a bijection between $\cim(R_{i}^{-}R_{i}^{+}(\mathcal{F}))$ and $\cim(\mathcal{F}')$.

$(2)$ If $M\in\cim(\mathcal{F})$, then $R^{+}_{i}(M)=0$ or $R^{+}_{i}(M)\in\cim(\mathcal{F'})$.

$(3)$ If $S_{i}\notin \mathcal{F}$, then $|\cim(\mathcal{F})|=|\cim(\mathcal{F}')|$. Otherwise, $|\cim(\mathcal{F})|-|\cim(\mathcal{F}')|=$
$$1-|\{M~|~M\in\cim(R_{i}^{-}R_{i}^{+}(\mathcal{F}))~{\text and}~M\notin\cim(\mathcal{F})\}|.$$

$(4)$  Assume that $M\in\cim(R_{i}^{-}R_{i}^{+}(\mathcal{F}))$. If $S_{i}\in \mathcal{F}$, the following are equivalent.
\begin{itemize}
\item [(i)] $M\in\cim(\mathcal{F})$.
\item [(ii)]  There exists no monomorphism $f:S_{i}\rightarrow M$ such that $\coker(f)\in\mathcal{F}$.
\end{itemize}
\end{proposition}
\begin{proof} (1) First, we show that $R_{i}^{+}(M)\in\cim(\mathcal{F}')$ for $M\in\cim(R_{i}^{-}R_{i}^{+}(\mathcal{F}))$. Note that, since $M\ncong S_{i}$, $R_{i}^{-}R_{i}^{+}(M)\cong M$. Take any exact sequence $$0\longrightarrow M_{1}\longrightarrow R_{i}^{+}(M)\stackrel{f}\longrightarrow M_{2}\longrightarrow0$$
in $\mathcal{F}'$. Then $R_{i}^{-}(M_{1})\longrightarrow M\stackrel{R_{i}^{-}(f)}\longrightarrow R_{i}^{-}(M_{2})\longrightarrow0$ is exact in $R_{i}^{-}(\mathcal{F}')=R_{i}^{-}R_{i}^{+}(\mathcal{F})$ since $R_{i}^{-}$ is right exact. Since $M\in\cim(R_{i}^{-}R_{i}^{+}(\mathcal{F}))$, either $\ker(R_{i}^{-}(f))=0$ or $R_{i}^{-}(M_{2})=0$. For the former, we have that $M\cong R_{i}^{-}(M_{2})$, and then $R_{i}^{+}(M)\cong R_{i}^{+}R_{i}^{-}(M_{2})\cong M_{2}$, thus $M_1=0$. For the latter, we have that $M_{2}=0$. Hence, $R_{i}^{+}(M)$ is a simple object in $\cim(\mathcal{F}')$.

Conversely, if $M\in\cim(\mathcal{F'})$, we prove that $R^{-}_{i}(M)\in\cim(R_{i}^{-}R_{i}^{+}(\mathcal{F}))$. Take an exact sequence $0\longrightarrow M_{1}\stackrel{f}\longrightarrow R_{i}^{-}(M)\stackrel{}\longrightarrow M_{2}\longrightarrow0$ in $R_{i}^{-}R_{i}^{+}(\mathcal{F})$. Noting that $S_{i}\notin R_{i}^{-}R_{i}^{+}(\mathcal{F})$ and $R^{+}_{i}$ is left exact, we have that $0\longrightarrow R_{i}^{+}(M_{1})\stackrel{R_{i}^{+}(f)}\longrightarrow M\longrightarrow R_{i}^{+}M_{2}$ is exact in $R_{i}^{+}R_{i}^{-}(\mathcal{F'})\subseteq\mathcal{F'}$ and thus either $\coker(R_{i}^{+}(f))=0$ or $R_{i}^{+}(M_{1})=0$. If $\coker(R_{i}^{+}(f))=0$, then $ R_{i}^{+}(M_{1})\cong M$ and $M_{1}\cong R_{i}^{-}R_{i}^{+}(M_{1})\cong R_{i}^{-}(M)$, thus $M_2=0$. If $R_{i}^{+}(M_{1})=0$, then $M_{1}=0$. Hence, $R^{-}_{i}(M)\in\cim(R_{i}^{+}R_{i}^{-}(\mathcal{F}))$.

$(2)$ If $M\cong S_{i}$, then $R^{+}_{i}(M)=0$. Otherwise, by (1), $R_{i}^{+}(M)\in\cim(\mathcal{F}')$.

$(3)$ It follows from $(1)$ and $(2)$.

$(4)$ $\rm (i)\Rightarrow(ii)$ Suppose that there is a monomorphism $S_{i}\rightarrow M$ such that $\coker(f)\in\mathcal{F}$, then we have that $S_{i}\cong M\in R_{i}^{-}R_{i}^{+}(\mathcal{F})$ since $M\in\cim(\mathcal{F})$.  This is a contradiction.

$\rm (ii)\Rightarrow(i)$ Take any exact sequence $0\longrightarrow M_{1}\stackrel{f}\longrightarrow M\stackrel{}\longrightarrow M_{2}\longrightarrow0$ in $\mathcal{F}$. We may assume that $M_{1}$ is indecomposable. Then $0\longrightarrow R_{i}^{+}(M_{1})\stackrel{R_{i}^{+}(f)}\longrightarrow R_{i}^{+}(M)\longrightarrow R_{i}^{+}(M_{2})$ is exact in $\mathcal{F}'$. Since $M\in\cim(R_{i}^{-}R_{i}^{+}(\mathcal{F}))$, by (1) we have that $R_{i}^{+}(M)\in\cim(\mathcal{F}')$. Thus, $R_{i}^{+}(M_{1})=0$ or $\coker(R_{i}^{+}(f))=0$. If $R_{i}^{+}(M_{1})=0$, then $R_{i}^{-}R_{i}^{+}(M_{1})=0$ and hence $M_{1}=0$ or $M_{1}\cong S_{i}$. By hypothesis, we know that $M_{1}\ncong S_{i}$ and thus $M_{1}=0$. If  $\coker(R_{i}^{+}(f))=0$, then $R_{i}^{+}(M_{1})\cong R_{i}^{+}(M)$ and thus $M\cong R_{i}^{-}R_{i}^{+}(M)\cong R_{i}^{-}R_{i}^{+}(M_{1})$. Since $M_{1}\ncong S_{i}$, we get that $M\cong M_{1}$ and $M_2=0$. This finishes the proof.
\end{proof}

\subsection{Coxeter groups }

Let us recall some combinatorial notions concerning Coxeter groups and quiver representations from \cite{En2} and \cite{Th}.

Let $Q$ be an $A$-type quiver with $n$ vertices and let $W$ be the {\em Coxeter group} generated by the {\em simple reflections} $s_{i}$ (cf. \cite[Section 1.2]{Th}). In this case, there is a group isomorphism from $W$ to the symmetric group $S_{n+1}$ on $n+1$ letters, which is given by sending $s_i$ to the adjacent transposition $(i~i+1)$. Hence, in what follows, we identify $W$ with $S_{n+1}$, and each $s_i$ with $(i~i+1)$.
We may use the {\em one-line notation} to represent elements of $S_{n+1}$. Explicitly, for any $w\in S_{n+1}$, we may write $w=w(1)w(2)\cdots w(n+1)$. Each element of $S_{n+1}$ can be expressed as a product of the simple reflections $s_i=(i~i+1)$. Such an expression for $w$ of minimal length is called {\em reduced}. We write
$l(w)$ for the length of a reduced expression for $w$. For any $w\in S_{n+1}$, we say $i\in\{1,\cdots,n\}$ is a {\em support} of $w$ if there exists some reduced expression of $w$ which contains $s_{i}$. Note that, if $i\in\supp(w)$, then  any reduced expression of $w$ contains $s_{i}$. We denote by $\supp(w)$ the set of all supports of $w$. Note that the support $\supp(e)$ of the unit element $e$ in $S_{n+1}$ is empty. For a subcategory $\mathcal{T}$ of $\rep(Q)$, we set $$\supp(\mathcal{T})=\{i~|~\text{there exists}~M\in\mathcal{T}~\text{such that}~M_{i}\neq0\}.$$
A {\em Coxeter element} of $S_{n+1}$ is an element $c\in S_{n+1}$ which is obtained as the product of all simple reflections $s_{1},\cdots s_{n}$ in some order. In particular, we can define a Coxeter element $c_{Q}\in S_{n+1}$ such that $s_{i}$ appears before $s_{j}$ in $c_{Q}$ whenever there exists an arrow  $j\rightarrow i$ in $Q$. Given a Coxeter element $c$ of $S_{n+1}$, we say that an element $w\in S_{n+1}$ is {\em $c$-sortable}  if there exists a reduced expression of the form $w=c^{0}c^{1}\cdots c^{t}$ such that each $c^{i}$ is a subword of $c$ satisfying $\supp(c^{0})\supset\supp(c^{1})\supset\cdots\supset\supp(c^{t})$.
For example, let $Q=1\leftarrow 2$, then $c_{Q}=s_{1}s_{2}$ and there are five $c_{Q}$-sortable elements $e$, $s_{1}$, $s_{2}$, $s_{1}s_{2}$ and $s_{1}s_{2}s_{1}$.

In the symmetric group $S_{n+1}$, we denote by $(i~j)$ for $1\leq i\neq j\leq n+1$ the transposition of the letters $i$ and $j$, and write $T$ for the set of all transpositions.
Since for any $1\leq i\neq j\leq n+1$, $(i~j)=(j~i)$ in $S_{n+1}$. So we always assume that each element in $T$ is written as $(i~j)$ with $i<j$.
Note that for a transposition $\sigma=(i~j)$ and $w\in S_{n+1}$, the one-line notation for $\sigma w$ can be obtained by interchanging the letters $i$ and $j$ in the one-line notation of $w$. Given an element $w\in S_{n+1}$, a transposition $\sigma\in T$ is called an {\em inversion} of $w$ if $l(\sigma w)<l(w)$ holds, and we denote by $\inv(w)$
the set of all inversions of $w$. By \cite[Section 6.1]{En2}, we have that $|\inv(w)|=l(w)$ and
$$\inv(w)=\{(i~j)\in T~|~\text{$j$ precedes $i$ in the one-line
notation for}~w\}.$$
We say that an inversion $\sigma$ of an element $w\in S_{n+1}$ is a {\em Bruhat inversion} of $w$ if it satisfies $l(\sigma w)=l(w)-1$, and we denote by $\Binv(w)$ the set of Bruhat inversions of $w$.
By \cite[Lemma 6.3]{En2}, we have that
$$\Binv(w)=\{(i~j)\in\inv(w)~|~\text{there exists no $l$ with $i<l<j$ such that}~(i~l),(l~j)\in\inv(w)\}.$$
For example, if $w=534216\in S_{6}$, then $$\inv(w)=\{(1~2),(1~3),(1~4),(1~5),(2~3),(2~4),(2~5),(3~5),(4~5)\}$$
and
$$\Binv(w)=\{(1~2),(2~3),(2~4),(3~5),(4~5)\}.$$
For a transposition $(i~j)\in T$, we define $M=M_{[i,j)}\in\ind(\rep(Q))$ as follows:
\begin{itemize}
\item For each vertex $a$ of $Q$, $M_{a}=k$ if $i\leq a<j$ and $M_{a}=0$ otherwise.
\item For each arrow $s\rightarrow t$ in $Q$, we put ${\rm id}_k:M_{s}\rightarrow M_{t}$ if $M_s=M_t=k$, and $0$ otherwise.
\end{itemize}
Clearly, the correspondence $(i~j)\mapsto M_{[i,j)}$ gives a bijection from $T$ to $\ind(\rep(Q))$.

For subsequent needs, we collect the following results.

\begin{theorem}\label{EnT} {\rm (\cite{En2}),\cite{Th})} Let $Q$ be an $A_{n}$ type quiver. For $w\in S_{n+1}$, define
$$\mathcal{F}(w)=\add\{M_{[i,j)}~|~(i~j)\in\inv(w)\}.$$

$(1)$ The correspondence $w\mapsto\mathcal{F}(w)$ gives a bijection between $c_{Q}$-sortable elements and torsion-free classes of $\rep(Q)$.

$(2)$ If $w$ is $c_{Q}$-sortable, then the correspondence $(i~j)\mapsto M_{[i,j)}$ gives a bijection between $\inv(w)$ and $\ind(\mathcal{F}(w))$, which also induces a bijection $\Binv(w)\rightarrow\cim(\mathcal{F}(w))$.

$(3)$ If $w$ is $c_{Q}$-sortable, then $\supp(w)$ coincides with  $\supp(\mathcal{F}(w))$.

$(4)$ If $w$ is $c_{Q}$-sortable, then $\mathcal{F}(w)$ satisfies  {\rm (JHP)} if and only if $|\supp(w)|=|\Binv(w)|$.

$(5)$  Let $Q$ be a linearly oriented quiver of type $A_{n}$. Then every torsion-free class in $\rep(Q)$  satisfies  {\rm (JHP)}.

$(6)$  If $f:M_{[l,l')}\rightarrow M_{[i,j)}$ is a monomorphism, then $i\leq l<l'\leq j$ and $\coker(f)\cong M_{[i,l)}\oplus M_{[l',j)}$. Here we set $M_{[s,t)}=0$ if $s=t$.

$(7)$ For any $w=w(1)\cdots w(n+1)\in S_{n+1}$, we have that $i\in\supp(w)$ if and only if $\max\{w(k)~|~k\leq i\}>i$.
\end{theorem}

\subsection{Torsion-free classes}
In this subsection, we always assume that $Q$ is an $A_{n}$ type quiver, $i$ is a sink of $Q$ and $Q'=\mu_{i}(Q)$; fix a torsion-free class $\mathcal{F}$ in $\rep(Q)$.
By Theorem \ref{EnT}(1), $\mathcal{F}=\mathcal{F}(w)$ for some $c_{Q}$-sortable element $w=w(1)\cdots w(n+1)\in S_{n+1}$. If $w$ has a reduced expression beginning with $s_{i}$, then $l(s_{i}w)=l(w)-1$, otherwise, $l(s_{i}w)=l(w)+1$. In what follows, we always assume that $l(s_{i}w)=l(w)-1$, i.e., there exists a reduced expression $w=s_{i}s_{t_{1}}\cdots s_{t_{k}}$. Observe that $\mathcal{F}'=R^{+}_{i}(\mathcal{F})$ is a torsion-free class in $ \rep(Q')$ which does not contain $S'_{i}$. Hence, $\mathcal{F}'=\mathcal{F}(w')$ for some  $c_{Q'}$-sortable element $w'$. Indeed, by the proof of \cite[Proposition 4.5]{Th}, we have that $\mathcal{F}=\mathcal{F}(s_{i}w')$. Since different elements of $S_{n+1}$ have different inversion sets, we obtain that $w=s_{i}w'$ and there is a reduced expression $w'=s_{t_{1}}\cdots s_{t_{k}}$. Thus,  if $i\notin\supp(w')$, then $\supp(w)=\supp(w')\cup\{i\}$, otherwise, $\supp(w)=\supp(w')$. We fix the elements $w$ and $w'$ as above in this subsection.

For convenience, for any $u\in S_{n+1}$, when we write $u=\cdots i\cdots j\cdots$, it means that $u^{-1}(i)<u^{-1}(j)$. Thus,
$$\Binv(u)=\{(i~j)\in\inv(u)~|~u=\cdots jl_{1}l_{2}\cdots l_{x}i\cdots  \text{satisfying}~\text{each}~l_{k}\notin[i,j])\}.$$

\begin{lemma}\label{m0}We have that $(i~i+1)\in\inv(w)$, i.e., $w=\cdots i+1\cdots i\cdots$.
\end{lemma}
\begin{proof} Since $l(s_{i}w)=l(w)-1$, we get that $i\in\supp(w)$. By Theorem \ref{EnT}(3), $i\in \supp(\mathcal{F})$. Thus, there exists $M\in \mathcal{F}$ such that $M_{i}\neq0$. Noting that $S_{i}$ is projective in $\rep(Q)$, we obtain that $\Hom(S_{i},M)\neq0$. Since $\mathcal{F}$ is a torsion-free class, we conclude that $S_{i}=M_{[i,i+1)}\in \ind(\mathcal{F})$. Therefore, by Theorem \ref{EnT}(2) we have that $(i~i+1)\in\inv(w)$.
\end{proof}
For the convenience of the following statements, we introduce the following notations. For any $u\in S_{n+1}$ and $1< i<n$, set
\begin{align*}
\alpha_{u}&=\{t\geq3~|~(2~t)\in\Binv(u)\}\\
&=\{t\geq3~|~u=\cdots tl_{1}\cdots l_{x}2\cdots  \text{satisfying}~\text{each}~l_{j}\notin[2,t]\},
\end{align*}
\begin{align*}
\beta_{u}^{i}&=\{t\leq i-1~|~(t~i)\in\Binv(u)~\text{and}~u^{-1}(i)<u^{-1}(i+1)<u^{-1}(t)\}\\
&=\{t\leq i-1~|~u=\cdots il_{1}\cdots i+1\cdots l_{x}t\cdots  \text{satisfying}~\text{each}~l_{j}\notin[t,i]\}
\end{align*}
and
\begin{align*}
\gamma_{u}^{i}&=\{t\geq i+2~|~(i+1~t)\in\Binv(u)~\text{and}~u^{-1}(t)<u^{-1}(i)<u^{-1}(i+1)\}\\
&=\{t\geq i+2~|~u=\cdots tl_{1}\cdots i\cdots l_{x} i+1\cdots  \text{satisfying}~l_{j}\notin[i+1,t]\}.
\end{align*}

\begin{lemma}\label{m11} Assume that $i=1$. Then we have that

$(1)$ $|\supp(w)|=|\supp(w')|$ if and only if $w(1)\neq2$ if and only if $|\cim(\mathcal{F})|-|\cim(\mathcal{F}')|=1-|\alpha_{w'}|$.

$(2)$ $|\supp(w)|=|\supp(w')|+1$ if and only if  $w(1)=2$ if and only if $|\cim(\mathcal{F})|-|\cim(\mathcal{F}')|=1$.
\end{lemma}
\begin{proof} We divide the proof into the following steps:

$(a)$ For any $M_{[s,t)}\in\rep(Q)$, by Proposition \ref{D11}(1), we have that
$$R^{+}_{1}(M_{[s,t)})=\left\{\begin{aligned}
   0, &  & \text{if}~s=1,t=2, \\
   M_{[2,t)}, &  & \text{if}~s=1,t\geq3,\\
   M_{[1,t)}, &  & \text{if}~s=2,\\
   M_{[s,t)}, &  & \text{if}~s\geq3.
  \end{aligned}
  \right.$$

$(b)$ Using Theorem \ref{EnT}(7) together with Lemma \ref{m0}, we have that $1\in\supp(w')\Leftrightarrow w'(1)\neq 1\Leftrightarrow w(1) \neq 2$. Thus,
 $$|\supp(w)|-|\supp(w')|=0\Leftrightarrow 1\in\supp(w')\Leftrightarrow w(1)\neq2$$
and
$$|\supp(w)|-|\supp(w')|=1\Leftrightarrow 1\notin\supp(w')\Leftrightarrow w(1)=2.$$

$(c)$ For any $M'\in\ind(\mathcal{F'})$, $M'=R^{+}_{1}(M)$ for some $M=M_{[s,t)}\in\mathcal{F}$. Then
\begin{align*}
M_{[s,t)}\in\cim(R_{1}^{-}R_{1}^{+}(\mathcal{F}))&\Leftrightarrow R^{+}_{1}(M_{[s,t)})=:M_{[a,b)}\in\cim(\mathcal{F'})~~~~&\text{(By~Proposition \ref{MM1}(1))}\\
&\Leftrightarrow (a~b)\in\Binv(w')~~~~&\text{(By Theorem \ref{EnT}(2))}
\end{align*}
$$\Leftrightarrow w'=\left\{\begin{aligned}
   \cdots tl_{1}\cdots l_{x}2\cdots  \text{satisfying each}~l_{j}\notin[2,t], &  & \text{if}~s=1,t\geq3, \\
   \cdots tl_{1}\cdots l_{x}1\cdots  \text{satisfying each}~l_{j}\notin[1,t], &  & \text{if}~s=2,\\
   \cdots tl_{1}\cdots l_{x}s\cdots  \text{satisfying each}~l_{j}\notin[s,t], &  & \text{if}~s\geq3.
  \end{aligned}
  \right.$$
Here the last equivalence follows from $(a)$.

$(d)$ Let $M:=M_{[s,t)}\in\cim(R_{1}^{-}R_{1}^{+}(\mathcal{F}))$. Clearly, $M_{1}=0$ if and only if $s\geq 2$. If $s\geq2$, then $M\in\cim(\mathcal{F})$ by Proposition \ref{MM1}(4). If $s=1$,  then there exists a monomorphism $f:S_{1}\rightarrow M$. Since $M\ncong S_{1}$, we have that $t\geq 3$. Thus, by Theorem \ref{EnT}(6), we have an exact sequence
$$0\longrightarrow M_{[1,2)}\stackrel{f}\longrightarrow M_{[1,t)}\longrightarrow M_{[2,t)}\longrightarrow0$$
in $\rep(Q)$. Then
\begin{align*}
M_{[1,t)}\in \cim(\mathcal{F})&\Leftrightarrow M_{[2,t)}\notin \mathcal{F}~~~~&\text{(By~Proposition \ref{MM1}(4))}\\
&\Leftrightarrow(2~t)\notin\inv(w)~~~~&\text{(By Theorem \ref{EnT}(2))}\\
&\Leftrightarrow w=12\cdots~\text{or}~w=2\cdots~~~~&\text{(By the definition of}~\inv(w))\\
&\Leftrightarrow w(1)=2~~~~&\text{(By Lemma \ref{m0}})
\end{align*}

$(e)$ By the definition of $\alpha_{w'}$ and $(c)$, $M_{[1,t)}\in\cim(R_{1}^{-}R_{1}^{+}(\mathcal{F}))\Leftrightarrow t\in\alpha_{w'}$. Hence, we conclude that
\begin{align*}
|\supp(w)|-|\supp(w')|=0&\Leftrightarrow w(1)\neq2~~~~&\text{(By~$(b)$})\\
&\Leftrightarrow M_{[1,t)}\notin\cim(\mathcal{F})~\text{for any}~t\in\alpha_{w'}&~\text{(By $(d)$})\\
&\Leftrightarrow |\cim(\mathcal{F})|-|\cim(\mathcal{F}')|=1-|\alpha_{w'}|.~~~~&\text{(By Proposition \ref{MM1}(3)})
\end{align*}
Similarly,
$$|\supp(w)|-|\supp(w')|=1\Leftrightarrow w(1)=2\Leftrightarrow M_{[1,t)}\in\cim(R_{1}^{+}R_{1}^{-}(\mathcal{F}))~\text{for any}~t\in\alpha_{w'}$$
$$\Leftrightarrow |\cim(\mathcal{F})|-|\cim(\mathcal{F}')|=1.$$
Therefore, we complete the proof.
\end{proof}

\begin{lemma}\label{m13} Assume that  $1<i<n$. Then we have that

$(1)$  $|\supp(w)|-|\supp(w')|=0$ if and only if $\max\{w'(k)~|~k\leq i\}> i$.

$(2)$ $|\supp(w)|-|\supp(w')|=1$ if and only if $\max\{w'(k)~|~k\leq i\}\leq i$.

$(3)$ $|\cim(\mathcal{F})|-|\cim(\mathcal{F}')|=1-(|\beta_{w'}^{i}|+|\gamma_{w'}^{i}|).$
\end{lemma}
\begin{proof} $(1)$-$(2)$ By Theorem \ref{EnT}(7),
 $$|\supp(w)|-|\supp(w')|=0\Leftrightarrow i\in\supp(w')\Leftrightarrow \max\{w'(k)~|~k\leq i\}>i$$
and
$$|\supp(w)|-|\supp(w')|=1\Leftrightarrow i\notin\supp(w')\Leftrightarrow\max\{w'(k)~|~k\leq i\}\leq i.$$

$(3)$ We divide the proof into the following steps:

$(a)$ For any $M_{[s,t)}\in\rep(Q)$, by Proposition \ref{D11}(1), we have that
$$R^{+}_{i}(M_{[s,t)})=\left\{\begin{aligned}
   0, &  & \text{if}~t=s+1=i+1, \\
   M_{[s,t)}, &  & \text{if}~t\leq i-1~\text{or}~s\geq i+2,\\
   M_{[s,i+1)}, &  & \text{if}~t=i,\\
   M_{[s,i)}, &  & \text{if}~t=i+1~\text{and}~s\leq i-1,\\
   M_{[s,t)}, &  & \text{if}~s\leq i-1~\text{and}~t\geq i+2,\\
    M_{[i,t)}, &  & \text{if}~s=i+1,\\
   M_{[i+1,t)}, &  & \text{if}~s=i~\text{and}~t\geq i+2.
  \end{aligned}
  \right.$$

$(b)$ Using Proposition \ref{MM1}(1) together with Theorem \ref{EnT}(2), we have that
$$M\in\cim(R_{i}^{-}R_{i}^{+}(\mathcal{F}))\Leftrightarrow R^{+}_{i}(M)=:M_{[a,b)}\in\cim(\mathcal{F'})\Leftrightarrow (a~b)\in\Binv(w').$$

Hence, by $(a)$, $M_{[s,t)}\in\cim(R_{i}^{+}R_{i}^{-}(\mathcal{F}))$ if and only if
$$w'=\left\{\begin{aligned}
   \cdots tl_{1}\cdots l_{x}s\cdots  \text{satisfying each}~l_{j}\notin[s,t], &  & \text{if}~t\leq i-1~\text{or}~s\geq i+2, \\
   \cdots i+1~l_{1}\cdots l_{x}s\cdots  \text{satisfying each}~l_{j}\notin[s,i+1], &  & \text{if}~t=i,\\
   \cdots il_{1}\cdots l_{x}s\cdots  \text{satisfying each}~l_{j}\notin[s,i], &  & \text{if}~t=i+1~\text{and}~s\leq i-1,\\
   \cdots tl_{1}\cdots l_{x}s\cdots  \text{satisfying each}~l_{j}\notin[s,t], &  & \text{if}~s\leq i-1~\text{and}~t\geq i+2,\\
   \cdots tl_{1}\cdots l_{x}i\cdots  \text{satisfying each}~l_{j}\notin[i,t], &  & \text{if}~s=i+1,\\
   \cdots tl_{1}\cdots l_{x}~i+1\cdots  \text{satisfying each}~l_{j}\notin[i+1,t], &  & \text{if}~s=i~\text{and}~t\geq i+2.
  \end{aligned}
  \right.$$

$(c)$ Let $M=M_{[s,t)}\in\cim(R_{i}^{-}R_{i}^{+}(\mathcal{F}))$. Then $M_{i}\neq0$ if and only if $s\leq i\leq t-1$. Now, we assume that $s\leq i\leq  t-1$, then there is an exact sequence
$$0\longrightarrow M_{[i,i+1)}\stackrel{}\longrightarrow M_{[s,t)}\longrightarrow M_{[s,i)}\oplus M_{[i+1,t)}\longrightarrow0$$
in $\rep(Q)$. By $(b)$, we need to discuss the following cases:

 $\mathbf{Case}$ 1: Assume that $t=i+1$ and $s\leq i-1$. In this case,
 \begin{align*}
M_{[s,t)}\notin \cim(\mathcal{F})&\Leftrightarrow M_{[s,i)}\in \mathcal{F}~~~~&\text{(By~Proposition \ref{MM1}(4))}\\
&\Leftrightarrow(s~i)\in\inv(w)~~~~&\text{(By Theorem \ref{EnT}(2))}\\
&\Leftrightarrow w=\cdots i+1\cdots i\cdots s\cdots~~~~&\text{(By~Lemma \ref{m0})}\\
&\Leftrightarrow w'=\cdots i\cdots i+1\cdots s\cdots.
\end{align*}
Recall that \begin{align*}
\beta_{w'}^{i}&=\{s\leq i-1~|~w'=\cdots il_{1}\cdots i+1\cdots l_{x}s\cdots  \text{satisfying}~\text{each}~l_{j}\notin[s,i]\}.
\end{align*}
Hence, for $a\leq i-1$, $M_{[a,i+1)}\in\cim(R_{i}^{-}R_{i}^{+}(\mathcal{F}))$, and $M_{[a,i+1)}\notin\cim(\mathcal{F})\Leftrightarrow a\in \beta_{w'}^{i}$.

$\mathbf{Case}$ 2: Assume that $s\leq i-1$ and $t\geq i+2$. Then
 \begin{align*}
M_{[s,t)}\notin \cim(\mathcal{F})&\Leftrightarrow M_{[s,i)},M_{[i+1,t)}\in \mathcal{F}~~~~&\text{(By~Proposition \ref{MM1}(4))}\\
&\Leftrightarrow(s~i),(i+1~t)\in\inv(w)~~~~&\text{(By Theorem \ref{EnT}(2))}\\
&\Leftrightarrow w=\cdots t\cdots i+1\cdots i\cdots s\cdots~~~~&\text{(By~Lemma \ref{m0})}\\
&\Leftrightarrow w'=\cdots t\cdots i\cdots i+1\cdots s\cdots.
\end{align*}
Since  $M_{[s,t)}\in\cim(R_{i}^{-}R_{i}^{+}(\mathcal{F}))$, by $(b)$, we have that
 $$w'=\cdots tl_{1}\cdots l_{x}s\cdots  \text{satisfying each}~l_{j}\notin[s,t].$$
Hence, we obtain that $M_{[s,t)}\in \cim(\mathcal{F})$.

$\mathbf{Case}$ 3: Assume that $s=i$ and $t\geq i+2$.  Then
\begin{align*}
M_{[s,t)}\notin \cim(\mathcal{F})&\Leftrightarrow M_{[i+1,t)}\in \mathcal{F}~~~~&\text{(By~Proposition \ref{MM1}(4))}\\
&\Leftrightarrow(i+1~t)\in\inv(w)~~~~&\text{(By Theorem \ref{EnT}(2))}\\
&\Leftrightarrow w=\cdots t\cdots i+1\cdots i\cdots ~~~~&\text{(By~Lemma \ref{m0})}\\
&\Leftrightarrow w'=\cdots t\cdots i\cdots i+1\cdots.
\end{align*}
Note that \begin{align*}
\gamma_{w'}^{i}&=\{t\geq i+2~|~w'=\cdots tl_{1}\cdots i\cdots l_{x} i+1\cdots  \text{satisfying~each}~l_{j}\notin[i+1~t]\}.
\end{align*}
Hence, for $b\geq i+2$, $M_{[i,b)}\in\cim(R_{i}^{-}R_{i}^{+}(\mathcal{F}))$ and $M_{[i,b)}\notin\cim(\mathcal{F})\Leftrightarrow b\in \gamma_{w'}^{i}$.

$(d)$ By $(b)$, $(c)$ and Proposition \ref{MM1}(3), we obtain that $$|\cim(\mathcal{F})|-|\cim(\mathcal{F}')|=1-(|\beta_{w'}^{i}|+|\gamma_{w'}^{i}|).$$
Therefore, we complete the proof.
\end{proof}
\subsection{(JHP) for torsion-free classes}
Given $w\in S_{n+1}$ and  $1\leq i<n$. Assume that $l(s_{i}w)=l(w)-1$. We define $\dag_{i}(w)$ and $\ddag_{i}(w)$ as follows:

$(1)$ If $i=1$, then
$$\dag_{i}(w):=\left\{\begin{aligned}
   0, &  & \text{if}~w(1)\neq 2, \\
   1, &  &\text{if}~w(1)=2
 \end{aligned}
  \right.$$
and
$$\ddag_{i}(w):=\left\{\begin{aligned}
   1-|\alpha_{s_{i}w}|, &  & \text{if}~w(1)\neq 2, \\
   1, &  &\text{if}~w(1)=2.
 \end{aligned}
  \right.$$
$(2)$ If $1<i<n$, then
$$\dag_{i}(w):=\left\{\begin{aligned}
   0, &  & \text{if}~\max\{s_{i}w(k)~|~k\leq i\}>i, \\
   1, &  &\text{if}~\max\{s_{i}w(k)~|~k\leq i\}\leq i
 \end{aligned}
  \right.$$
and
$$\ddag_{i}(w):=1-(|\beta_{s_{i}w}^{i}|+|\gamma_{s_{i}w}^{i}|).$$

\begin{lemma}\label{mm2} Let $Q$ be an $A_{n}$ type quiver and $1\leq i< n$ be a sink of $Q$. Let $w$ be a $c_{Q}$-sortable element. If $l(s_{i}w)=l(w)-1$, then
$$\dag_{i}(w)=|\supp(w)|-|\supp(s_{i}w)|$$
and
$$\ddag_{i}(w)=|\Binv(w)|-|\Binv(s_{i}w)|.$$
\end{lemma}
\begin{proof} It follows from Theorem \ref{EnT}(2), Lemma \ref{m11} and Lemma \ref{m13}.
\end{proof}

Let $Q$ be any quiver of type $A_{n}$ and let $w$ be a $c_{Q}$-sortable element. Then there  exists $1\leq j_{1},\cdots, j_{s}\leq n-1$ such that $R^{+}_{j_{s}}\cdots R^{+}_{j_{1}}(Q)=1\rightarrow 2\rightarrow\cdots\rightarrow n$. We define the {\em reflection sequence} $j_{t_{1}},\cdots, j_{t_{x}}$ of $w$ such that
\begin{flalign*}&t_{1}=\min\{y~|~l(s_{j_{y}}w)=l(w)-1~\text{for}~1\leq y\leq s\}\\
&\vdots\\
&t_{x-1}=\min\{y~|~l(s_{j_{y}}s_{j_{t_{x}-2}}\cdots s_{j_{t_{1}}}w)=l(s_{j_{t_{x-2}}}\cdots s_{j_{t_{1}}}w)-1~\text{for}~t_{x-2}+1\leq y\leq s\}\\
&t_{x}=\min\{y~|~l(s_{j_{y}}s_{j_{t_{x-1}}}\cdots s_{j_{t_{1}}}w)=l(s_{j_{t_{x-1}}}\cdots s_{j_{t_{1}}}w)-1~\text{for}~t_{x-1}+1\leq y\leq s\}\end{flalign*}
and
$$l(s_{j_{y}}s_{j_{t_{x}}}\cdots s_{j_{t_{1}}}w)\neq l(s_{j_{t_{x}}}\cdots s_{j_{t_{1}}}w)-1~\text{for any}~t_{x}+1\leq y\leq s.$$
In this case, we set
$$\dag_{[j_{s}\cdots j_{1}]}(w):=\dag_{j_{t_{1}}}(w)+\cdots +\dag_{j_{t_{x}}}(s_{j_{t_{x-1}}}\cdots s_{j_{t_{1}}}w)$$
and
$$\ddag_{[j_{s}\cdots j_{1}]}(w):=\ddag_{j_{t_{1}}}(w)+\cdots +\ddag_{j_{t_{x}}}(s_{j_{t_{x-1}}}\cdots s_{j_{t_{1}}}w).$$
If there exits no reflection sequence of $w$, we set $\dag_{[j_{s}\cdots j_{1}]}(w)=\ddag_{[j_{s}\cdots j_{1}]}(w)=0$.

Now we are in position to present the main result of this section as the following
\begin{theorem}\label{main6} Let $Q$ be any quiver of type $A_{n}$ and let $w$ be a $c_{Q}$-sortable element. Take $1\leq j_{1},\cdots, j_{s}< n$ such that $R^{+}_{j_{s}}\cdots R^{+}_{j_{1}}(Q)=1\rightarrow 2\rightarrow\cdots\rightarrow n=:Q'$. Then $\mathcal{F}(w)$ satisfies  {\rm (JHP)} if and only if $\dag_{[j_{s}\cdots j_{1}]}(w)=\ddag_{[j_{s}\cdots j_{1}]}(w)$.
\end{theorem}
\begin{proof} We write $W_{<s_{i}>}$
for the Coxeter group generated by the reflections $s_{j}$ with $j\neq i$. For any sink $q$ of $Q$, $s_{q}c_{Q}$ is a Coxeter element in $W_{<s_{i}>}\cong S_{n}$. By \cite[Theorem 4.1(1)]{Th}, if $l(s_{q}w)=l(w)+1$, then $w$ is $s_{q}c_{Q}$-sortable; noting that $s_{q}c_{Q}s_{q}=c_{\mu_{q}(Q)}$, we obtain that $w$ is a $c_{\mu_{q}(Q)}$-sortable element. If $l(s_{q}w)=l(w)-1$, by \cite[Theorem 4.1(2)]{Th}, $s_{q}w$ is a $c_{\mu_{q}(Q)}$-sortable element. Hence, if there exists no reflection sequence of $w$, then $w$ is a $c_{Q'}$-sortable element and thus $\mathcal{F}(w)$ satisfies  {\rm (JHP)} by Theorem \ref{EnT}(5), in this case, $\dag_{[j_{s}\cdots j_{1}]}(u)=\ddag_{[j_{s}\cdots j_{1}]}(u)=0$.

Now we assume that the reflection sequence of $w$ exists, denoted by $j_{t_{1}},\cdots, j_{t_{x}}$. As mentioned above, we know that $s_{j_{t_{x}}}\cdots s_{j_{t_{1}}}w$ is a $c_{Q'}$-sortable element. By the proof of \cite[Proposition 4.5]{Th}, we have that  $R^{+}_{j_{t_{x}}}\cdots R^{+}_{j_{t_{1}}}(\mathcal{F}(w))=\mathcal{F}(s_{j_{t_{x}}}\cdots s_{j_{t_{1}}}w)$ is a torsion-free class in $\rep(Q')$. By Lemma \ref{mm2}, we get that
\begin{flalign*}&|\supp(w)|-|\supp(s_{j_{t_{x}}}\cdots s_{j_{t_{1}}}w)|\\&=|\supp(w)|-|\supp(s_{j_{t_{1}}}w)|+|\supp(s_{j_{t_{1}}}w)|\\&\quad-|\supp(s_{j_{t_{2}}}s_{j_{t_{1}}}w)|+|\supp(s_{j_{t_{2}}}s_{j_{t_{1}}}w)|-\cdots-|\supp(s_{j_{t_{x-1}}}\cdots s_{j_{t_{1}}}w)|\\&\quad+|\supp(s_{j_{t_{x-1}}}\cdots s_{j_{t_{1}}}w)|-|\supp(s_{j_{t_{x}}}\cdots s_{j_{t_{1}}}w)|\\
&=\dag_{j_{t_{1}}}(w) +\cdots +\dag_{j_{t_{x}}}(s_{j_{t_{x-1}}}\cdots s_{j_{t_{1}}}w)\\&=\dag_{[j_{s}\cdots j_{1}]}(w)\end{flalign*}
and similarly,
$$|\Binv(w)|-|\Binv(s_{j_{t_{x}}}\cdots s_{j_{t_{1}}}w)|=\ddag_{[j_{s}\cdots j_{1}]}(w).$$
Using Theorem \ref{EnT}(4-5), we obtain that $\mathcal{F}(w)$ satisfies  {\rm (JHP)} if and only if
$$\supp(w)-\supp(s_{j_{t_{x}}}\cdots s_{j_{t_{1}}}w)=\Binv(w)-\Binv(s_{j_{t_{x}}}\cdots s_{j_{t_{1}}}w)$$
if and only if $\dag_{[j_{s}\cdots j_{1}]}(u)=\ddag_{[j_{s}\cdots j_{1}]}(u)$. Therefore, we complete the proof.
\end{proof}

As a consequence, we can give a characterization of the torsion-free classes satisfying {\rm (JHP)} in a combinatorial way.
\begin{corollary}\label{BB} Let $Q$ be the quiver $1\leftarrow 2\rightarrow\cdots\rightarrow n$ and $w\in S_{n+1}$. Then $\mathcal{F}(w)$ is a torsion-free class satisfying {\rm (JHP)} if and only if the following conditions hold.
\begin{itemize}
  \item [(i)] $w\neq \cdots 2\cdots k\cdots 1\cdots$ for any $2<k<n+1$.
  \item [(ii)]  $w\neq \cdots k\cdots i\cdots j$ for any $j\in\{3,4,\cdots, n\}$ and $i<j< k$.
  \item [(iii)]  One of the following conditions holds:
  \begin{itemize}
  \item [(a)] $w=\cdots 1\cdots 2\cdots$.
  \item [(b)] If $w=\cdots 2\cdots 1\cdots$, then $w(1)=2$ or $w(1)\neq2$ and $|\alpha_{s_{1}w}|=1.$
  \end{itemize}
  \end{itemize}
\end{corollary}

\begin{proof}By \cite[Theorem 4.12]{Re}, $\mathcal{F}(w)$ is a torsion-free class if and only if $w$ satisfies the conditions {\rm (i)} and {\rm(ii)}. Now, we assume that $\mathcal{F}(w)$ is a torsion-free class. If $w=\cdots 1\cdots 2\cdots$, then $S_{1}\notin\mathcal{F}(w)$. By the proof of Lemma \ref{m0}, we can obtain that $1\notin\supp(w)$. In this case, we get that $l(s_{1}w)=l(w)+1$ and thus $\dag_{[1]}(w)=\ddag_{[1]}(w)=0$.

If $w=\cdots 2\cdots 1\cdots$, by {\rm (i)}, $w=\cdots 21\cdots$ and $s_{1}w=\cdots 12\cdots$. Suppose that $(i~j)\in\inv(s_{1}w)$, then $s_{1}w=\cdots j\cdots i\cdots$. If $i=1$, then $j>2$ and thus $w=\cdots j\cdots 12\cdots$. If $i=2$, we  have that $w=\cdots j\cdots 21\cdots$. If $i>3$, it is clear that $w=\cdots j\cdots i\cdots$. Hence, $\inv(s_{1}w)\subseteq\inv(w)$. Since $l(w)=|\inv(w)|\geq |\inv(s_{1}w)|=l(s_{1}w)$, we obtain that  $l(s_{1}w)=l(w)-1$. By Theorem \ref{main6}, $\mathcal{F}(w)$ is a torsion-free class satisfying {\rm (JHP)} if and only if
$$\dag_{[1]}(w)=\dag_{1}(w)=|\supp(w)|-|\supp(s_1w)|=\ddag_{[1]}(w)=\ddag_{1}(w)=|\cim(\mathcal{F}(w))|-|\cim(\mathcal{F}(s_1w))|.$$
By Lemma \ref{m11}, we obtain that $\dag_{1}(w)=0\Leftrightarrow \ddag_{1}(w)=1-|\alpha_{s_1w}|\Leftrightarrow w(1)\neq2$, and $\dag_{1}(w)=1\Leftrightarrow \ddag_{1}(w)=1\Leftrightarrow w(1)=2$. This finishes the proof.
\end{proof}

We finish this section with a straightforward example illustrating Theorem \ref{main6}.
\begin{example} Let $Q$ be the quiver $1\rightarrow 2\leftarrow 3$ and $R^{+}_{1}R^{+}_{2}(Q)=1\rightarrow2\rightarrow3=:Q'$. Then $c_{Q}=s_{2}s_{1}s_{3}=s_{2}s_{3}s_{1}=3142.$
Take $w=s_{2}s_{3}s_{1}s_{2}=3412$ and it is a $c_{Q}$-sortable element. Since $s_{2}w=s_{3}s_{1}s_{2}=2413$ and $s_{1}s_{2}w=s_{3}s_{2}=1423$, we obtain that $2,1$ is a reflection sequence of $w$. By Lemma \ref{mm2}, it is easy to see that $\dag_{2}(w)=0$ and $\dag_{1}(s_{2}w)=1$. Recall that
$$\beta_{s_{2}w}^{2}=\{t\leq1~|~s_{2}w=\cdots 2l_{1}\cdots 3\cdots l_{x}t\cdots  \text{satisfying each}~l_{j}\notin[t,2]\}$$
and $$\gamma_{s_{2}w}^{2}=\{t\geq4~|~s_{2}w=\cdots tl_{1}\cdots 2\cdots l_{x} 3\cdots  \text{satisfying each}~l_{j}\notin[3,t]\},$$
which are clearly both empty.
Hence, $\ddag_{2}(w)=1-(|\beta_{s_{2}w}^{2}|+|\gamma_{s_{2}w}^{2}|)=1$. Since $(s_{2}w)(1)=2$, by definition, we have that $\ddag_{1}(s_{2}w)=1$. Hence, $\dag_{[12]}(w)=\dag_{2}(w)+\dag_{1}(s_{2}w)=1$ and $\ddag_{[12]}(w)=\ddag_{2}(w)+\ddag_{1}(s_{2}w)=2$. Therefore, by Theorem \ref{main6}, {\rm (JHP)} does not hold in $\mathcal{F}(w)$.

Let us list 14 $c_{Q}$-sortable elements and the corresponding reflection sequences, $c_{Q'}$-sortable elements, $\dag_{12}(w)$ and $\ddag_{12}(w)$ as follows:
$$
\begin{tabular}{|p{3.3cm}|p{3.5cm}|p{3.3cm}|c|p{1.5cm}|}
\hline
$c_{Q}$-sortable $w$ & Reflection sequence&$c_{Q'}$-sortable $w'$&$\dag_{[12]}(w)$ & $\ddag_{[12]}(w)$\\
\hline
e=1234& absence & e=1234&0 &0\\
\hline
$s_{1}$=2134& absence & $s_{1}$=2134& 0&0\\
\hline
$s_{2}$=1324&$2$ &e=1234 & 1&1\\
\hline
$s_{3}$=1243& absence&$s_{3}$=1243 &0 &0\\
\hline
$s_{1}s_{3}$=2143&$1$ &$s_{3}$=1243 & 1&1\\
\hline
$s_{1}s_{3}$=2143& $1$& $s_{3}$=1243& 1&1\\
\hline
$s_{2}s_{1}$=3124& $2$,$1$& e=1234& 2&2\\
\hline
$s_{2}s_{1}s_{2}$=3214& $2$,$1$& $s_{2}$=1324&1 &1\\
\hline
$s_{2}s_{3}s_{2}$=1432& $2$& $s_{3}s_{2}$=1423& 0&0\\
\hline
$s_{2}s_{3}s_{1}$=3142& $2$,$1$&$s_{3}=1243$ & 2&2\\
\hline
$s_{2}s_{3}s_{1}s_{2}$=3412&  $2$,$1$& $s_{3}s_{2}=1423$& 1&2\\
\hline
$s_{2}s_{3}s_{1}s_{2}s_{1}$=4312& $2$,$1$& $s_{3}s_{2}s_{1}=4123$ & 0&0\\
\hline
$s_{2}s_{3}s_{1}s_{2}s_{3}$=3421& $2$,$1$&$s_{3}s_{2}s_{3}=1432$ &1 &1\\
\hline
$s_{2}s_{3}s_{1}s_{2}s_{3}s_{1}$=4321& $2$,$1$&$s_{3}s_{2}s_{1}s_{3}=4123$ &0 &0\\
\hline
\end{tabular}
$$
Therefore, for all $c_{Q}$-sortable elements $w$ except 3412, the torsion-free classes $\mathcal{F}(w)$ satisfy {\rm (JHP)}.
\end{example}

%%%%%%%%%%%%%%%%%%%%%%%%%%%%%%%%%%%%%%%%%%%%%%%%%%%%%%%%%%%%%%%%%%%%%%%%%%%%%%%%%%%%%%%%%%%%%%%%%%%%%%%%%%%%%%%%%%%%%%%%%%%%


\begin{thebibliography}{99}
\bibitem{Ar} T. Araya,  Exceptional sequences over path algebras of type $A_{n}$ and non-crossing spanning trees, 	
Algebr. Represent. Theory {\bf 16} (2013), 239--250.
\bibitem{Asa} S. Asai, Semibricks, Int. Math. Res. Not. {\bf 16} (2020), 4993--5054.
\bibitem{Aus} M. Auslander, I. Reiten, S.O. Smal{\o}, Representation Theory of Artin Algebras, Cambridge Studies in Advanced Mathematics, vol.36, Cambridge University Press, Cambridge, 1995.
\bibitem{Be}  A. Berenstein, J. Greenstein, Primitively generated Hall algebras, Pac. J. Math. {\bf 281} (2016), 287--331.
\bibitem{BHL} T. Br\"{u}stle, S. Hassoun, D. Langford, S. Roy, Reduction of exact structures, J. Pure Appl. Algebra {\bf 224}(4) (2020), 106212.
\bibitem{Br} T. Br\"{u}stle, S. Hassoum, A. Tattar, Intersections, sums, and the Jordan-H\"{o}lder property for exact categories, J. Pure Appl. Algebra {\bf 225}(11) (2021), 106724.
\bibitem{En1} H. Enomoto, Schur's lemma for exact categories implies abelian, J. Algebra {\bf 584} (2021), 260--269.
\bibitem{En3} H. Enomoto, Monobrick, a uniform approach to torsion-free classes and wide subcategories, Adv. Math {\bf 393} (2021), 108113.
\bibitem{En2} H. Enomoto, The Jordan-H\"{o}lder property and Grothendieck monoids of exact categories, Adv. Math {\bf 369} (2022), 108167.
\bibitem{Na} H. Nakaoka, Y. Palu, Extriangulated categories, Hovey twin cotorsion pairs and model structures, Cah. Topol. G\'{e}om. Diff\'{e}r. Cat\'{e}g. {\bf 60}(2) (2019), 117--193.
\bibitem{Re} N. Reading, Cluster, Coxeter-sortable elements and noncrossing partitions, Trans. Amer. Math. soc. {\bf 359} (2007), 5931--5958.
\bibitem{Ringel} C.M. Ringel, Representations of K-species and bimodules, J. Algebra {\bf 41}(2) (1976), 269--302.
\bibitem{St} B. Stenstr\"{o}m, Rings of Quotients: An Introduction to Methods of Ring Theory, Springer-Verlag, New York-Heidelberg, 1975.
\bibitem{Th} H. Thomas, Coxeter groups and quiver representations, in: Surveys in Representation Theory of Algebras, in: Contemp. Math., vol. 716, Amer. Math. Soc., Providence, RI, 2018.
\bibitem{Wa} L. Wang, J. Wei, H. Zhang, Semibricks in extriangulated categories, Comm. Algebra {\bf 49} (2021), 5247--5262.
\end{thebibliography}
\end{document}